\documentclass[final,3p,times]{elsarticle}
\usepackage{amssymb}
\usepackage{amsthm}
\usepackage{latexsym}
\usepackage{amsmath}
\usepackage{color}
\usepackage{graphicx}
\usepackage{indentfirst}
\usepackage{mathrsfs}
\usepackage{pdfsync}
\usepackage{hyperref}
\usepackage{cleveref}
\hypersetup{hypertex=true,
	colorlinks=true,
	linkcolor=blue,
	anchorcolor=blue,
	citecolor=blue}
\allowdisplaybreaks
\usepackage[utf8]{inputenc}

\renewcommand{\thefootnote}{\arabic{footnote}}
\newtheorem{theorem}{\color{black}\indent Theorem}
\newtheorem{lemma}{\color{black}\indent Lemma}[section]

\newtheorem{definition}{\color{black}\indent Definition}[section]
\newtheorem{remark}{\color{black}\indent Remark}[section]
\newtheorem{corollary}{\color{black}\indent Corollary}[section]

\newcommand\blfootnote[1]{%
		\begingroup
		\renewcommand\thefootnote{}\footnote{#1}%
		\addtocounter{footnote}{-1}%
		\endgroup
	}
\journal{\ }

\begin{document}

\begin{frontmatter}
\title{Universal frequency-preserving KAM persistence via modulus of continuity}

\author{{ \blfootnote{$^{*}$Corresponding author at: School of Mathematics, Jilin University, Changchun 130012, People’s Republic of China} Zhicheng Tong$^{a}$ \footnote{ E-mail address : tongzc20@mails.jlu.edu.cn}
		,~ Yong Li$^{a,b,*}$}  \footnote{E-mail address : liyong@jlu.edu.cn}\\
	{$^{a}$College of Mathematics, Jilin University,} {Changchun 130012, P. R. China.}\\
	{$^{b}$School of Mathematics and Statistics, and Center for Mathematics and Interdisciplinary Sciences, \\Northeast Normal University,}
	{Changchun, 130024, P. R. China.}
}

\begin{abstract}
In this paper, we study the persistence and remaining regularity of KAM invariant torus under sufficiently small perturbations of a Hamiltonian function together with its derivatives, in sense of finite smoothness with modulus of continuity, as a generalization of classical H\"{o}lder continuous circumstances. To achieve this goal, we extend the Jackson approximation theorem to the case of modulus of continuity, and establish a corresponding regularity theorem adapting to the new iterative scheme. Via these tools, we establish a KAM theorem with sharp differentiability hypotheses, which asserts that the persistent torus keeps prescribed universal Diophantine frequency unchanged 
and reaches  the regularity for persistent KAM torus beyond H\"older's type.
\end{abstract}


\end{frontmatter}

\section{Introduction}
The KAM theory mainly concerns the preservation of invariant tori of a Hamiltonian function $ H(y) $  under small perturbations (i.e., $H(y) \to H\left( {x,y,\varepsilon } \right) $ of freedom $ n \in \mathbb{N}^+ $ with $ \varepsilon>0 $ sufficiently small), which has a history of more than sixty years.  See, for instance, Kolmogorov and Arnold \cite{R-9,R-10,R-11}, Moser \cite{R-12,R-13}, P\"oschel \cite{Po1,Po2} and etc. As is known to all, for frequency $ \omega  = {H_y}\left( {y} \right) $ of the unperturbed system, we often require it to satisfy the following classical Diophantine condition (or be of Diophantine class $ \tau $)
\begin{equation}\label{dio}
| {\langle {{\tilde k},\omega } \rangle } | \geq \alpha_ *{ | {{\tilde k }} |}^{-\tau} ,\;\;\forall 0 \ne \tilde k \in {\mathbb{Z}^n}
\end{equation}
with respect to $ \tau \geq n-1 $ and some $ \alpha_ *>0 $, where $ |\tilde k|: = \sum\nolimits_{j = 1}^n {|{{\tilde k}_j}|}  $. Otherwise, the torus may break no matter how small the perturbation is. Furthermore, to ensure  the KAM persistence one also is interested in the minimal order of derivatives required for $ H\left( {x,y,\varepsilon } \right) $. Much effort has been devoted on this problem in terms of H\"older continuity, including constructing counterexamples and reducing the differentiability hypotheses. For some classic foundational work, see  Moser \cite{J-3}, Jacobowitz \cite{R-2}, Zehnder \cite{R-7,R-8}, Mather \cite{R-55},  Herman \cite{M1,M2}, Salamon \cite{salamon} and etc. It is worth mentioning that, very recently  P\"oschel \cite{Po3} obtained a  KAM theorem on $ n $-dimensional torus (without action variables) based on a frequency being of Diophantine class $ \tau=n-1 $ in \eqref{dio}.  Specially, he pointed out that the  derivatives of order $ n $ need not be continuous, but rather $ L^2 $ in a certain strong sense.

Back to our concern on Hamiltonian systems with action-angular variables, it is always conjectured that the minimum regularity requirement for the Hamiltonian function $ H $ is at least $ C^{2n} $. Along with the idea of Moser, the best known H\"older case $ C^{\ell} $ with  $ \ell >2\tau+2>2n $ has been established by Salamon in \cite{salamon}, where the prescribed frequency is of Diophantine class $ \tau>n-1 $ in \eqref{dio} (with full Lebesgue measure and thus reveals the universality of the KAM persistence), and the remaining regularity of the KAM torus is also showed to be H\"older's type. More precisely, the resulting solutions are of class $ C^m $ with $ 0< m<2\ell-2\tau-2 $, and  the function whose graph is the invariant torus is of class  $ C^{m+\tau+1} $. Besides, the differentiability hypotheses is sharp due to the counterexample work of Herman  \cite{M1,M2} et al., which will be explained later in \cref{subsubsub}. In the aspect of H\"older's type, see Bounemoura \cite{Bounemoura} and Koudjinan \cite{Koudjinan} for some new developments.  Strictly weaker than H\"older continuity, Albrecht \cite{Chaotic} proved a KAM theorem via a strong Diophantine frequency of class $ \tau=n-1 $ in \eqref{dio}, which claimed  that $ C^{2n} $  plus certain  modulus of continuity $ \varpi $ satisfying the classical Dini condition
\begin{equation}\label{cdini}
	\int_0^1 {\frac{{\varpi \left( x \right)}}{x}dx}  <  + \infty
\end{equation}
 is enough for the KAM persistence.  Such  strong Diophantine frequencies are continuum many and  form a set of zero Lebesgue measure, see details from  \cite{Polecture}, therefore the corresponding KAM preservation is usually said to be non-universal. To the best of our knowledge,  there is no other work on KAM via only modulus of continuity except for \cite{Chaotic}.  Back to our concern on universal KAM persistence in this paper, the best result so far still requires $ C^{2n} $ plus certain H\"older continuity depending on the Diophantine nonresonance. It is therefore natural that ones should consider the following questions:

\begin{itemize}
\item \textit{Can H\"{o}lder smoothness in Salamon's KAM  be further weakened into a general form of modulus of continuity?}

\item \textit{If the invariant KAM torus persists, then what kind of smoothness does the torus have (H\"{o}lder continuity, or more general  modulus of continuity)?}

\item \textit{Could the prescribed universal Diophantine frequency to be kept unchanged?}

\item \textit{Does there exist a Dini type integrability condition similar to \eqref{cdini} that reveals the explicit relation between nonresonance and regularity?}
\end{itemize}


 To answer the above questions, there are at least four difficulties to overcome. Firstly,  note that the Jackson approximation theorem for classical H\"{o}lder continuity is no longer valid at present, hence it  must be developed to approximate the perturbed Hamiltonian function $  H\left( {x,y,\varepsilon } \right) $ in the sense of modulus of continuity,  as a crucial step.  Secondly, it is also basic how to establish a corresponding regularity iteration lemma to study the regularity of the invariant torus and the solution beyond H\"older's type. Thirdly, we need to set up a new KAM iterative scheme and prove its uniform convergence via these tools. Fourthly, it is somewhat difficult to extract an equilibrium integrability condition of nonresonance and regularity from KAM iteration, as well as further touch the remaining regularity. Indeed, to achieve the main result \cref{theorem1}, we apply \cref{Theorem1}  to construct a series of analytic approximations to $  H\left( {x,y,\varepsilon } \right) $ with modulus of continuity, and prove the persistence and regularity of invariant torus via a modified  KAM iteration as well as a generalized Dini type condition. It should be pointed out that our results still admit sharpness on differentiability $ C^{2n} $ due to Herman's work \cite{M1,M2}, where he considered the nonexistence of an  invariant curve for an annulus mapping being of H\"older  regularity $ {C^{3 - \epsilon }} $ with any $ \epsilon $ close to $ 0^+ $, i.e., $ C^{2n}=C^{4} $ minus arbitrary H\"older continuity cannot admit KAM persistence when $ n=2 $.

 As some new efforts, our \cref{theorem1} applies to a wide range, including non-universal and universal KAM persistence, and reveals the integral  relation between regularity and nonresonance. Apart from above, it is well known that small divisors must lead to the loss of regularity, and our approach gives  general estimates  of the  KAM remaining regularity without H\"older continuity for the first time. Particularly, as a direct  application, our \cref{theorem1} could deal with  the case of general modulus of continuity for $ H\left( {x,y,\varepsilon }\right) $, such as Logarithmic H\"{o}lder continuity case, i.e., for all $ 0 < \left| {x - \xi } \right| + \left| {y - \eta } \right| \leq 1/2 $,
\begin{displaymath}
	\left| {{\partial ^\alpha }H\left( {x,y,\varepsilon } \right) - {\partial ^\alpha }H\left( {\xi ,\eta ,\varepsilon } \right)} \right| \leq \frac{c}{{{{\left( { - \ln \left( {\left| {x - \xi } \right| + \left| {y - \eta } \right|} \right)} \right)}^\lambda }}}
\end{displaymath}
with respect to all $ \alpha  \in {\mathbb{N}^{2n}}$ with $ \left| \alpha  \right| = {2n } $, where $ n \geq 2 $,   $\lambda>1$,  $c, \varepsilon>0 $ are sufficiently small,  $ \left( {x,y} \right) \in {\mathbb{T}^n} \times G $ with $ {\mathbb{T}^n}: = {\mathbb{R}^n}/ \mathbb{Z}^n $, and $ G \subset {\mathbb{R}^n} $ is a connected closed set with interior points. See \cref{section6} for more details.

This paper is organized as follows. In  \cref{section2}, we first introduce some notions and properties for modulus of continuity, and establish a Jackson type approximation theorem based on them (the proof will be postponed to \cref{JACK}). Then we state our main results in this paper. Namely, considering that the higher-order derivatives of Hamiltonian function $ H $ with respect to the action-angular variables are only continuous, we present a KAM theorem (\cref{theorem1}) with sharp differentiability hypotheses under certain assumptions, involving a generalized Dini type integrability condition \textbf{(H1)}. The applications of this theorem are given in  \cref{section6}, including non-universal (\cref{simichaos2}) and universal (\cref{Holder,lognew}) KAM persistence. For the former, we reach a conclusion similar to that in \cite{Chaotic}. As to the latter, we provide H\"{o}lder and H\"{o}lder plus Logarithmic H\"{o}lder circumstances, aiming to show the importance and universality of \cref{theorem1}. In particular,  an explicit Hamiltonian function $ H $ is constructed, which cannot be studied by KAM theorems for finite smoothness via classical H\"{o}lder continuity, but the work generalized in this paper can be applied.  \cref{section4} provides the proof of \cref{theorem1} and is mainly divided into two parts: the first part deals with the modified KAM steps via only modulus of continuity, while the second part is devoted to giving an iteration  theorem (\cref{t1}) on regularity,  which is used to analyze the remaining smoothness for the persistent invariant torus. \cref{proofsimichaos2,8,pflognew} present the proof of \cref{simichaos2,Holder,lognew} in \cref{section6}, respectively.

\section{Statement of results}
\label{section2}
We first give some notions, including the modulus of continuity along with the norm based on it,  the  semi separability which will be used in  \cref{Theorem1}, as well as the weak homogeneity which will appear in \cref{theorem1}.

Denote by $ |\cdot| $ the sup-norm in $ \mathbb{R}^d $ and the dimension $ d \in \mathbb{N}^+ $ may vary throughout this paper. We formulate that in the limit process, $ f_1(x)=\mathcal{O}^{\#}\left(f_2(x)\right) $ means there are absolute positive constants $ \ell_1 $ and $ \ell_2 $ such that $ {\ell _1}{f_2}\left( x \right) \leq {f_1}\left( x \right) \leq {\ell _2}{f_2}\left( x \right) $, and $ f_1(x)=\mathcal{O}\left(f_2(x)\right)  $ implies that there exists an absolute positive constant $ \ell_3 $ such that $ |f_1(x)| \leq \ell_3 f_2(x) $, and finally $ f_1(x)\sim f_2(x)  $ indicates that $ f_1(x) $ and $ f_2(x) $ are equivalent.

\begin{definition}\label{d1}
	Let $ \varpi (t)>0 $ be a nondecreasing continuous function on the interval $ \left( {0,\delta } \right] $ with respect to some $ \delta  >0 $ such that $ \mathop {\lim }\limits_{x \to {0^ + }} \varpi \left( x \right) = 0 $ and $ \mathop {\overline {\lim } }\limits_{x \to {0^ + }} x/{\varpi }\left( x \right) <  + \infty  $.  Next, we define the following semi norm and norm for a continuous function $ f $ on $ {\mathbb{R}^n} $ ($f\in C^0$, for short)
	\begin{equation}\notag
		{\left[ f \right]_\varpi }: = \mathop {\sup }\limits_{x,y \in {\mathbb{R}^n},\;0 < \left| {x - y} \right| \leq \delta } \frac{{\left| {f\left( x \right) - f\left( y \right)} \right|}}{{\varpi \left( {\left| {x - y} \right|} \right)}},\;\;{\left| f \right|_{{C^0}}}: = \mathop {\sup }\limits_{x \in {\mathbb{R}^n}} \left| {f\left( x \right)} \right|.
	\end{equation}
	We say that  $ f $ is of $ C_{k,\varpi}  $ continuous if $ f $ has partial derivatives $ {{\partial ^\alpha }f} $ for $ \left| \alpha  \right| \leq k \in \mathbb{N} $ and satisfies
	\begin{equation}\label{k-w}
		{\left\| f \right\|_\varpi }: = \sum\limits_{\left| \alpha  \right| \leq k} {\left( {{{\left| {{\partial ^\alpha }f} \right|}_{{C^0}}} + {{\left[ {{\partial ^\alpha }f} \right]}_\varpi }} \right)}  <  + \infty .
	\end{equation}
	Denote by $ {C_{k,\varpi }}\left( {{\mathbb{R}^n}} \right) $ the space composed of all functions $ f $ satisfying \eqref{k-w}.
\end{definition}

Such a function  $ \varpi $ is usually referred to as the modulus of continuity of $ f $. It can be seen that the well-known Lipschitz continuity and H\"{o}lder continuity are special cases in the above definition. In particular, for $ 0<\ell \notin \mathbb{N}^+ $, we denote by $ f \in {C^\ell }\left( {{\mathbb{R}^n}} \right) $ the function space in which the higher derivatives in $ \mathbb{R}^n $ are H\"{o}lder continuous, i.e., the modulus of continuity is of the form $ \varpi_{\mathrm{H}}^{\{\ell\}}(x)\sim x^{\ell} $, where $ \{\ell\} \in (0,1)$ denotes the fractional part of $ \ell $. As a generalization of classical H\"{o}lder continuity, we define the Logarithmic H\"{o}lder continuity with index $ \lambda  > 0 $, where $ \varpi_{\mathrm{LH}}^{\lambda} \left( x \right) \sim 1/{\left( { - \ln x} \right)^\lambda } $, and we omit the the range $ 0 < x   \ll 1 $ without causing ambiguity.
\begin{remark}
	For $ f:{\mathbb{R}^n} \to \Omega  \subset {\mathbb{R}^d} $ with a modulus of continuity $ \varpi $, we modify the above designation to $ {C_{k,\varpi }}\left( {{\mathbb{R}^n},\Omega } \right) $.
\end{remark}
\begin{remark}\label{rema666}
	It is well known that a mapping defined on a bounded connected closed set in a finite dimensional space must have a modulus of continuity, see \cite{Herman3}. For example, for a function $ f(x) $ defined on $ [0,1]  \subset {\mathbb{R}^1} $, it automatically admits a modulus of continuity
	\[{\omega _{f,\delta }}\left( x \right): = \mathop {\sup }\limits_{y \in \left[ {0,1} \right],0 < \left| {x - y} \right| \leq \delta } \left| {f\left( x \right) - f\left( y \right)} \right|.\]
\end{remark}

\begin{definition}\label{d5}
	Let $ {\varpi _1} $ and $ {\varpi _2} $ be modulus of continuity on interval $ \left( {0,\delta } \right] $. We say that  $ {\varpi _1} $ is  weaker (strictly weaker) than $ {\varpi _2} $ if $ \mathop {\overline\lim }\limits_{x \to {0^ + }} {\varpi _2}\left( x \right)/{\varpi _1}\left( x \right) <+\infty $ ($ =0 $).
\end{definition}
\begin{remark}\label{strict}
	Obviously any modulus of continuity is weaker than Lipschitz's type, and  the Logarithmic H\"{o}lder's type $ \varpi_{\mathrm{LH}}^{\lambda} \left( x \right) \sim 1/{\left( { - \ln x} \right)^\lambda } $ with any $ \lambda  > 0 $ is strictly weaker than arbitrary  H\"{o}lder's type $ \varpi_{\mathrm{H}}^{\alpha}\left( x \right) \sim {x^\alpha } $ with any $ 0 < \alpha  < 1 $.
\end{remark}

\begin{definition}[Semi separability]\label{d2}
	We say that $ \varpi  $ in \cref{d1} is   semi separable, if for $ x \geq 1 $, there holds
	\begin{equation}\label{Ox}
		\psi \left( x \right): = \mathop {\sup }\limits_{0 < r < \delta /x} \frac{{\varpi \left( {rx} \right)}}{{\varpi \left( r \right)}} = \mathcal{O}\left( x \right),\;\;x \to  + \infty .
	\end{equation}
\end{definition}
\begin{remark}\label{Remarksemi}
	Semi separability directly leads to $ \varpi \left( {rx} \right) \leq \varpi \left( r \right)\psi \left( x \right) $ for $ 0 < rx \leq \delta  $, which will be used in the proof of the  Jackson type  \cref{Theorem1} via only modulus of continuity.
\end{remark}

\begin{definition}[Weak homogeneity] \label{weak}
A modulus of continuity $ \varpi $ is said to admit weak homogeneity, if for fixed $ 0<a<1 $, there holds
\begin{equation}\label{erfenzhiyi}
	\mathop {\overline {\lim } }\limits_{x \to {0^ + }} \frac{{\varpi \left( x \right)}}{{\varpi \left( {ax} \right)}} <  + \infty .
\end{equation}
\end{definition}

It should be emphasized that semi separability and weak homogeneity are universal hypotheses.	The H\"older and Lipschitz type  automatically admit them. Many modulus of continuity weaker than the H\"older one are semi separable and also admit weak homogeneity, e.g., for the Logarithmic H\"{o}lder's type $ \varpi_{\mathrm{LH}}^{\lambda} \left( x \right) \sim 1/{\left( { - \ln x} \right)^\lambda } $ with any  $ \lambda  > 0 $, one verifies that $ \psi \left( x \right) \sim {\left( {\ln x} \right)^\lambda } = \mathcal{O}\left( x \right) $ as $ x \to +\infty $ in \eqref{Ox}, and $ \mathop {\overline {\lim } }\limits_{x \to {0^ + }} {\varpi_{\mathrm{LH}}^{\lambda}}\left( x \right)/{\varpi_{\mathrm{LH}}^{\lambda}}\left( {ax} \right) = 1 <  + \infty  $ with all $ 0<a<1 $ in \eqref{erfenzhiyi}. See more implicit examples in \cref{Oxlemma,ruotux},  in particular, \textit{it is pointed out that a convex modulus of continuity  naturally possesses these two properties.}

Next, we give a Jackson type approximation theorem beyond H\"older's type and some related corollaries based on  \cref{d1,d2}, their  proof will be postponed to   \cref{JACK,proofcoro1,proofcoco2}, respectively.

\begin{theorem}\label{Theorem1}
	There is a family of convolution operators
	\begin{equation}\notag
		{S_r}f\left( x \right) = {r^{ - n}}\int_{{\mathbb{R}^n}} {K\left( {{r^{ - 1}}\left( {x - y} \right)} \right)f\left( y \right)dy} ,\;\;0 < r \leq 1,
	\end{equation}
	from $ {C^0}\left( {{\mathbb{R}^n}} \right) $ into the space of entire functions on $ {\mathbb{C}^n} $ with the following property. For every $ k \in \mathbb{N} $, there exists a constant $ c\left( {n,k} \right)>0 $ such that, for every $ f \in {C_{k,\varpi }}\left( {{\mathbb{R}^n}} \right) $ with a semi separable modulus of continuity $ \varpi $, every multi-index $ \alpha  \in {\mathbb{N}^n} $ with $ \left| \alpha  \right| \leq k  $, and every $ x \in {\mathbb{C}^n} $ with $ \left| {\operatorname{Im} x} \right| \leq r $, we have
	\begin{equation}\label{3.2}
		\left| {{\partial ^\alpha }{S_r}f\left( x \right) - {P_{{\partial ^\alpha }f,k - \left| \alpha  \right|}}\left( {\operatorname{Re} x;\mathrm{i}\operatorname{Im} x} \right)} \right| 		\leq c\left( {n,k} \right){\left\| f \right\|_\varpi }{r^{k - \left| \alpha  \right|}}\varpi(r),
	\end{equation}
	where the Taylor polynomial $ P $ is defined as follows
	\[{P_{f,k}}\left( {x;y} \right) := \sum\limits_{\left| \beta  \right| \leq k} {\frac{1}{{\alpha !}}{\partial ^\beta }f\left( x \right){y^\alpha }}. \]
	Moreover,  $ {{S_{r}}f} $ is real analytic whenever $ f $ is real valued.
\end{theorem}


As a direct consequence of \cref{Theorem1}, we give the following \cref{coro1,coro2}. 
These results have been  widely used in H\"older's case, see for instance, \cite{Koudjinan,salamon}.

\begin{corollary}\label{coro1}
	The approximation function $ {{S_r}f\left( x \right)} $ in  \cref{Theorem1} satisfies
	\begin{equation}\notag
		\left| {{\partial ^\alpha }\left( {{S_r}f\left( x \right) - f\left( x \right)} \right)} \right| \leq c_*{\left\| f \right\|_\varpi }{r^{k - \left| \alpha  \right|}}\varpi(r)
	\end{equation}
and
	\begin{equation}\notag
	\left| {{\partial ^\alpha }{S_r}f\left( x \right)} \right| \leq {c^ * }{\left\| f \right\|_\varpi }
\end{equation}
	for  $ x \in \mathbb{C}^n $ with $ \left| {\operatorname{Im} x} \right| \leq r $, $ |\alpha| \leq k $, where $ c_* = c_*\left( {n,k} \right) >0$ and $ {c^ * } = {c^ * }\left( {n,k,\varpi } \right) >0$ are some universal constants.
\end{corollary}

\begin{corollary}\label{coro2}
	If the function $ f\left( x \right) $ in  \cref{Theorem1} also satisfies that the period of each variables $ {x_1}, \ldots ,{x_n} $ is $ 1 $ and the integral on $ {\mathbb{T}^n} $ is zero, then the approximation function $ {S_r}f\left( x \right) $ also satisfies these properties.
\end{corollary}


We are now in a position to give the frequency-preserving KAM theorem via only modulus of continuity in this paper. Before this, let's start with our parameter settings.
Let $ n \geq 2$ (degree of freedom),  $\tau  \geq n - 1$ (Diophantine index), $  2\tau  + 2 \leq k \in {\mathbb{N}^ + } $ (differentiable order) and a sufficiently large number $ M>0 $ be given. Consider a Hamiltonian function $ H(x,y):{\mathbb{T}^n} \times G \to \mathbb{R} $ with $ {\mathbb{T}^n}: = {\mathbb{R}^n}/ \mathbb{Z}^n $, and $ G \subset {\mathbb{R}^n} $ is a connected closed set with interior points. It follows from \cref{rema666} that $ H $ automatically has a modulus of continuity $ \varpi $. In view the comments below \cref{weak}, we assume that $ \varpi $ admits semi separability (\cref{d2}) and weak homogeneity (\cref{weak}) without loss of generality.  Besides, we make the following assumptions:
\begin{itemize}
\item[\textbf{(H1)}] Integrability condition for modulus of continuity: Assume that $ H \in {C_{k,\varpi }}\left( {{\mathbb{T}^n} \times G} \right) $ with the above modulus of continuity $ \varpi $. In other words, $ H $ at least has derivatives of order $ k $, and the highest derivatives admit the regularity of $ \varpi $. Moreover, $ \varpi $ satisfies the Dini type  integrability condition
\begin{equation}\label{Dini}
	\int_0^1 {\frac{{\varpi \left( x \right)}}{{{x^{2\tau  + 3 - k}}}}dx}  <  + \infty .
\end{equation}

\item[\textbf{(H2)}] Boundedness and nondegeneracy:
\begin{equation}\notag
	{\left\| H \right\|_\varpi } \leq M,\;\;\left| {{{\left( {\int_{{\mathbb{T}^n}} {{H_{yy}}\left( {\xi ,0} \right)d\xi } } \right)}^{ - 1}}} \right| \leq M.
\end{equation}

\item[\textbf{(H3)}] Diophantine condition: For some $ \alpha_ *  > 0 $, the frequency $ \omega  \in {\mathbb{R}^n} $ satisfies
\begin{equation}\notag
	| {\langle {{\tilde k},\omega } \rangle } | \geq \alpha_ *{ | {{\tilde k }} |}^{-\tau} ,\;\;\forall 0 \ne \tilde k \in {\mathbb{Z}^n},\;\; |\tilde k|: = \sum\limits_{j = 1}^n {|{{\tilde k}_j}|}   .
\end{equation}

\item[\textbf{(H4)}] KAM smallness: There holds
\begin{align}\label{T1-2}
	&\sum\limits_{\left| \alpha  \right| \leq k} {\left| {{\partial ^\alpha }\Big( {H\left( {x,0} \right) - \int_{{\mathbb{T}^n}} {H\left( {\xi ,0} \right)d\xi } } \Big)} \right|{\varepsilon ^{\left| \alpha  \right|}}}  \notag \\
	+ &\sum\limits_{\left| \alpha  \right| \leq k - 1} {\left| {{\partial ^\alpha }\left( {{H_y}\left( {x,0} \right) - \omega } \right)} \right|{\varepsilon ^{\left| \alpha  \right| + \tau  + 1}}}  \leq M{\varepsilon ^k}\varpi \left( \varepsilon  \right)
\end{align}
for every $ x \in \mathbb{R}^n $ and some constant $ 0 < \varepsilon  \leq {\varepsilon ^ * } $.

\item[\textbf{(H5)}]    Criticality: For $ \varphi_i(x):=x^{k-(3-i)\tau-1}\varpi(x) $ with $ i=1,2 $, there exist critical $ k_i^*\in \mathbb{N}^+ $ such that
\[\int_0^1 {\frac{{{\varphi _i}\left( x \right)}}{{{x^{k_i^ *+1 }}}}dx}  <  + \infty ,\;\;\int_0^1 {\frac{{{\varphi _i}\left( x \right)}}{{{x^{k_i^ *  + 2}}}}dx}  =  + \infty .\]
\end{itemize}

Let us make some comments.
\begin{itemize}
\item[\textbf{(C1)}] There seems to be a large number of assumptions above, but they are important conditions abstracted from the H\"{o}lder continuous case, and we have to do so in order to give the KAM theorem in the case of only modulus of continuity. However, 
    some of such conditions, e.g. \textbf{(H2)}-\textbf{(H3)}, are classical, while some are ordinary. 	
	
\item[\textbf{(C2)}] In view of  \cref{rema666},  $ H $ automatically admits a modulus of continuity. The Dini type  integrability condition \eqref{Dini} in  \textbf{(H1)} is a direct generalization of H\"older's type, which can be seen in \cref{Holder}. Interestingly, it becomes the classical Dini condition \eqref{cdini} if $ \tau=n-1 $ and $ k=2\tau+2=2n $.

\item[\textbf{(C3)}] There is a large family of modulus of continuity  satisfying the classical Dini condition \eqref{cdini}, such as the Logarithmic H\"{o}lder's type $ \varpi_{\mathrm{LH}}^{\lambda} \left( x \right) \sim 1/{\left( { - \ln x} \right)^\lambda } $ with $ \lambda>1 $, and  even more complicated case: the generalized Logarithmic H\"older's type
\begin{equation}\label{mafan}
	\varpi_{\mathrm{GLH}}^{\varrho,\lambda} \left( x \right) \sim \frac{1}{{(\ln (1/x))(\ln \ln (1/x)) \cdots {{(\underbrace {\ln  \cdots \ln }_\varrho (1/x))}^\lambda }}}
\end{equation}
with any $ \varrho \in \mathbb{N}^+ $ and $ \lambda>1 $. In particular, $ \varpi_{\mathrm{LH}}^{\lambda}(x) \sim \varpi_{\mathrm{GLH}}^{1,\lambda}(x)$. Note that the above $ \lambda>1 $ cannot degenerate to $ 1 $, otherwise the Dini integral \eqref{cdini} diverges.

\item[\textbf{(C4)}]   According to the properties of Banach algebra, for the H\"older's type, it is assumed that \textbf{(H4)} only needs the term of $ \left| \alpha  \right| = 0 $, and does not need higher-order derivatives to satisfy the condition. However, for general modulus of continuity, it seems not easy to establish the corresponding Banach algebraic properties, we thus add higher-order derivatives in \textbf{(H4)}. 
     Sometimes they can be removed correspondingly.

\item[\textbf{(C5)}] The existence of $ k_i^* $ in \textbf{(H5)}  is directly guaranteed by \textbf{(H1)}, actually this assumption is proposed to investigate the higher regularity of the persistent KAM torus, that is, the regularity  to $ C^{k_i^*} $ plus certain modulus of continuity. In general, given an explicit modulus of continuity $ \varpi $, such $ k_i^* $ in \textbf{(H5)} are automatically  determined by using asymptotic analysis, see \cref{section6}.
\end{itemize}

Finally, we state the following frequency-preserving KAM theorem under sharp differentiability via only modulus of continuity:

\begin{theorem}[Main Theorem]\label{theorem1}
	Assume \textbf{(H1)}-\textbf{(H4)}. Then there is a solution
	\[x = u\left( \xi  \right),\;\;y = v\left( \xi  \right)\]
	of the following equation with the operator $ D: = \sum\limits_{\nu  = 1}^n {{\omega _\nu }\frac{\partial }{{\partial {\xi _\nu }}}}  $
	\begin{equation}\notag
		Du = {H_y}\left( {u,v} \right),\;\;Dv =  - {H_x}\left( {u,v} \right),
	\end{equation}
	such that $ u\left( \xi  \right) - \xi  $ and $ v\left( \xi  \right) $ are of period $ 1 $ in all variables, where $ u $ and $ v $ are at least $ C^1 $.
	
	 In addition, assume \textbf{(H5)}, then there exist $ {\varpi _i} $ ($ i=1,2 $) such that $ u \in {C_{k_1^ * ,{\varpi _1}}}\left( {{\mathbb{R}^n},{\mathbb{R}^n}} \right) $ and $ v \circ {u^{ - 1}} \in {C_{k_2^ * ,{\varpi _2}}}\left( {{\mathbb{R}^n},G} \right) $. Particularly, $ {\varpi _i} $ can be determined as follows
\begin{equation}\label{varpii}
	{\varpi _ i }\left( \gamma \right) \sim \gamma \int_{L_i\left( \gamma  \right)}^\varepsilon  {\frac{{\varphi_i \left( t \right)}}{{{t^{{k_i^*} + 2}}}}dt}  = {\mathcal{O}^\# }\left( {\int_0^{L_i\left( \gamma  \right)} {\frac{{\varphi_i \left( t \right)}}{{{t^{{k_i^*} + 1}}}}dt} } \right),\;\;\gamma  \to {0^ + },
\end{equation}
where $ L_i(\gamma) \to 0^+ $  are some functions such that the second relation in \eqref{varpii} holds for $ i=1,2 $.	
\end{theorem}

\begin{remark}
	We call such a solution $x=u(\xi),y=v(\xi)$ the KAM one. 	
\end{remark}

\begin{remark}
	With the same as in \cite{salamon}, the unperturbed systems under consideration might be non-integrable (e.g., $ H = \left\langle {\omega ,y} \right\rangle  + \left\langle {A\left( x \right)y,y} \right\rangle  +  \cdots  $),  and the KAM persistence is of frequency-preserving. The main difference from \cite{salamon} is that the regularity of the high-order derivatives and the derived smoothness for persistent torus is weakened to only modulus of continuity from the H\"{o}lder's type.
\end{remark}
\begin{remark}
Actually  \cref{theorem1} provides a method for determining $ \varpi_i $ with $ i=1,2 $, see \eqref{varpii}. For the prescribed modulus of continuity to  Hamiltonian, such as the H\"older and Logarithmic H\"older type, we have to use asymptotic analysis to derive the concrete continuity of the KAM torus in \cref{section6}.
\end{remark}

As mentioned forego, the H\"older's type $ H \in  C^{\ell}(\mathbb{T}^n,G) $ with $ \ell>2\tau+2 $ (where $ \tau>n-1$ is the Diophantine exponent) is always regarded as the critical case. Let $ k=[\ell] $. Then $ k=2\tau+2=2n $ ($ \tau=n-1 $ at present) seems to be the critical case in our setting, and our Dini type integrability condition \eqref{Dini} becomes the classical Dini condition \eqref{cdini}! But it should be noted that, such Diophantine frequencies with $ \tau=n-1 $ can only form a set of zero Lebesgue measure  and are therefore not enough to represent almost all frequencies. In other words, for universal KAM persistence, we may have to require the generalized Dini  condition in \textbf{(H1)}, which reveals the deep  relationship between the \textit{irrationality} for frequency $ \omega $, \textit{order} and \textit{continuity} of  the highest derivatives for the Hamiltonian $ H $. Obviously, if the highest differentiable order $ k $ of $ H $ satisfies $ k\geq 2\tau+3 $ or even larger, then \textbf{(H1)} will become trivial because $ \varpi $ does not have a singularity at $ 0 $. But our KAM theorem still makes sense, because the regularity of the persistent torus will also increase.

\section{Applications}
\label{section6}
In this section, we show certain detailed regularity  about KAM torus such as H\"{o}lder and Logarithmic H\"{o}lder ones etc. Denote by $ \{a\} $ and $ [a] $ the fractional part and the integer part  of $ a\geq0 $, respectively. It should be emphasized that the Dini type  integrability condition \eqref{Dini} in \textbf{(H1)} is easy to verify, that is, the KAM persistence is easy to obtain. However, some techniques of asymptotic analysis are needed to investigate the specific regularity of KAM torus, which is mainly reflected in the selection of functions $ L_i (\gamma)$ ($ i=1,2 $) in \eqref{varpii}. In particular, we will explicitly see the degree of regularity loss caused by small divisors, see for instance, \cref{simichaos2,Holder,lognew} and the example shown in \cref{explicitexa}.

We apply our \cref{theorem1} from two different perspectives. In \cref{subnon}, for the minimum regularity $ C^{2n} $ that is critical under our approach,  we investigate KAM preservation in the sense of zero Lebesgue measure (corresponds to non-universal), i.e., first let $ k=2n $, then determine Diophantine nonresonance $ \tau=n-1 $; while in  \cref{subun}, for the given Diophantine nonresonance $ \tau>n-1 $ of  full Lebesgue measure in advance  (corresponds to universal), we study the minimum regularity requirement under our method. In what follows, the modulus of continuity under consideration are always convex near $ 0^+ $ and therefore automatically admit semi separability as well as weak homogeneity which we forego.

\subsection{Non-universal KAM persistence}\label{subnon}
 Focusing  on non-universal KAM persistence for Hamiltonian systems with action-angular variables  of freedom $ n $, Albrecht \cite{Chaotic} proved that $ C^{2n} $ plus certain modulus of continuity satisfying the classical Dini condition \eqref{cdini} for regularity requirement   is enough. The frequencies he used are of Diophantine class $ \tau=n-1 $ in \textbf{(H3)}, i.e., of zero Lebesgue measure. However, it is still interesting to study the remaining regularity of the KAM torus, which is still unknown so far. By applying \cref{theorem1} we directly obtain the following \cref{simichaos} similar to that in \cite{Chaotic}, therefore the proof is omitted here. To illustrate our results, we provide an explicit example in \cref{simichaos2}, and the proof  will be postponed to \cref{proofsimichaos2}.

\begin{theorem}\label{simichaos}
Let $ k=2n $ and $ \tau=n-1 $ be given. Assume that \textbf{(H1)} \textbf{(H2)}, \textbf{(H3)} and \textbf{(H4)} hold with a convex modulus of continuity $ \varpi $. That is, the Hamiltonian $ H $ only has  derivatives of order $ 2n $,  the prescribed frequency is of Diophantine  class $ n-1 $, and \textbf{(H1)} turns to the classical Dini condition \eqref{cdini}. Then the KAM persistence in \cref{theorem1} could be admited.
\end{theorem}

\begin{theorem}\label{simichaos2}
In view of Comment \text{(C3)}, let the modulus of continuity in \cref{simichaos} be of the generalized Logarithmic H\"{o}lder's type in \eqref{mafan}, i.e.,
\begin{equation}\label{mafan2}
	\varpi_{\mathrm{GLH}}^{\varrho,\lambda} \left( x \right) \sim \frac{1}{{(\ln (1/x))(\ln \ln (1/x)) \cdots {{(\underbrace {\ln  \cdots \ln }_\varrho (1/x))}^\lambda }}}
\end{equation}
with $ \varrho \in \mathbb{N}^+ $ and $ \lambda>1 $. Then the remaining regularity in \cref{theorem1} is $ u \in {C_{1 ,{\varpi _1}}}\left( {{\mathbb{R}^n},{\mathbb{R}^n}} \right) $ and $ v \circ {u^{ - 1}} \in {C_{n,{\varpi _2}}}\left( {{\mathbb{R}^n},G} \right) $, where
\begin{equation}\label{rem}
{\varpi _1}\left( x \right) \sim {\varpi _2}\left( x \right) \sim \frac{1}{{{{(\underbrace {\ln  \cdots \ln }_\varrho (1/x))}^{\lambda  - 1}}}}.
\end{equation}
\end{theorem}
\begin{remark}
Particularly \eqref{mafan2} reduces to the Logarithmic H\"older's type $ \varpi_{\mathrm{LH}}^{\lambda}(x) \sim 1/{(-\ln x)^\lambda} $ with $ \lambda>1 $ as long as $ \varrho=1 $. As can be seen that, the remaining regularity in \eqref{rem} is much weaker than that in \eqref{mafan2}, and it is indeed very weak if $ \lambda>1 $ is sufficiently close to $ 1 $ (but cannot degenerate to $ 1 $, see Comment \textbf{(C3)}), because the explicit modulus of continuity in \eqref{rem} tends to $ 0 $ quite slowly as $ x \to 0^+ $.
\end{remark}

\subsection{Universal KAM persistence}\label{subun}
In this subsection, we always assume that the prescribed Diophantine frequencies $ \omega $  are of full Lebesgue measure, that is, $ \tau >n-1 $ in \textbf{(H3)}. Note that for fixed $ n $, the parameter $ \tau  $ might be very large, and the frequencies being of  Diophantine class $ \tau  $ are at least continuum many. Under such setting, the known minimum regularity requirement for Hamiltonian $ H $ is H\"older's type $ C^{\ell} $ with $ \ell>2\tau +2 $, see Salamon \cite{salamon} and \cref{Holder} below. Interestingly, if one considers weaker modulus of continuity, such as $ C^{2\tau+2} $ plus  Logarithmic H\"older's type, the above regularity could be weakened, see our new  \cref{lognew}.

\subsubsection{H\"{o}lder continuous case}\label{subsubsub}
\begin{theorem}\label{Holder}
Let $ H \in C^{\ell} (\mathbb{T}^n,G) $ with $ \ell>2\tau +2 $, where  $ \ell  \notin \mathbb{N}^+ $, $ \ell-\tau \notin \mathbb{N}^+$ and $ \ell-2\tau \notin \mathbb{N}^+ $. That is, $ H $ is of $ C_{k,\varpi} $ with $ k=[\ell] $ and $ \varpi(x)\sim \varpi_{\mathrm{H}}^{\ell}(x)\sim x^{\{\ell\}} $.	Assume \textbf{(H2)}, \textbf{(H3)} and \textbf{(H4)}. Then there is a solution $ x = u\left( \xi  \right),y = v\left( \xi  \right) $	of the following equation with the operator $ D: = \sum\limits_{\nu  = 1}^n {{\omega _\nu }\frac{\partial }{{\partial {\xi _\nu }}}}  $
	\begin{equation}\notag
		Du = {H_y}\left( {u,v} \right),\;\;Dv =  - {H_x}\left( {u,v} \right)
	\end{equation}
	such that $ u\left( \xi  \right) - \xi  $ and $ v\left( \xi  \right) $ are of period $ 1 $ in all variables. In addition, $ u \in {C^{\ell-2\tau-1}}\left( {{\mathbb{R}^n},{\mathbb{R}^n}} \right) $ and $ v \circ {u^{ - 1}} \in {C^{\ell-\tau-1}}\left( {{\mathbb{R}^n},G} \right) $.
\end{theorem}

\cref{Holder} has been completely proved in \cite{salamon}. Significantly, the differentiability hypotheses under consideration  is sharp, i.e., it is close to the optimal one as in \cite{M1,M2}, where Herman gave a counterexample about the nonexistence of an invariant curve for an annulus mapping of $ {C^{3 - \epsilon }} $ with $ 0 < \epsilon  \ll 1 $ corresponds to the case $ n=2,\ell=4-\varepsilon $ in our setting, which implies the sharpness of \cref{Holder}. See more from \cite{R-55,F-1}.

\subsubsection{H\"{o}lder plus Logarithmic H\"{o}lder continuous case}
To show different modulus of continuity  weaker than H\"older's type, we establish the following \cref{lognew}. 
 One will see later that \cref{lognew} employs  more complicated asymptotic analysis than \cref{simichaos2}, and interestingly, the remaining regularity $ \varpi_1 $ and $ \varpi_2 $ admit different forms. In fact, \cref{lognew} can completely contain the case of \cref{simichaos2}, that is, $ \tau=n-1 $, and $ \varrho=1 $ in \eqref{mafan2}. However, in order to distinguish the full Lebesgue measure and zero Lebesgue measure of Diophantine nonresonance, we show them separately.
\begin{theorem}\label{lognew}
Let $ \tau>n-1 $ be given and let $ H\in C_{[2\tau+2], \varpi} $, where $ \varpi \left( x \right) \sim {x^{\{2\tau+2\}}}/{\left( { - \ln x} \right)^\lambda } $ with $ \lambda  > 1 $. Assume \textbf{(H2)}, \textbf{(H3)} and \textbf{(H4)}. That is, $ H $ is of $ C^{k} $ plus the above $ \varpi $ with $k= [2\tau+2] $. Then there is a solution $ x = u\left( \xi  \right),y = v\left( \xi  \right) $	of the following equation with the operator $ D: = \sum\limits_{\nu  = 1}^n {{\omega _\nu }\frac{\partial }{{\partial {\xi _\nu }}}}  $
\begin{equation}\notag
	Du = {H_y}\left( {u,v} \right),\;\;Dv =  - {H_x}\left( {u,v} \right)
\end{equation}
such that $ u\left( \xi  \right) - \xi  $ and $ v\left( \xi  \right) $ are of period $ 1 $ in all variables. In addition, letting
\[{\varpi _1}\left( x \right) \sim \frac{1}{{{{\left( { - \ln x} \right)}^{\lambda  - 1}}}} \sim \varpi _{\mathrm{LH}}^{\lambda  - 1}\left( x \right),\]
and
\[	{\varpi _2}\left( x \right) \sim \left\{ \begin{aligned}
	&{\frac{1}{{{{\left( { - \ln x} \right)}^{\lambda  - 1}}}} \sim \varpi _{\mathrm{LH}}^{\lambda  - 1}\left( x \right)},&n-1<\tau  \in {\mathbb{N}^ + } \hfill, \\
	&{\frac{{{x^{\left\{ \tau  \right\}}}}}{{{{\left( { - \ln x} \right)}^\lambda }}} \sim {x^{\left\{ \tau  \right\}}}\varpi _{\mathrm{LH}}^\lambda \left( x \right)},&n-1<\tau  \notin {\mathbb{N}^ + } \hfill, \\
\end{aligned}  \right.\]
one has that $ u \in {C_{1 ,{\varpi _1}}}\left( {{\mathbb{R}^n},{\mathbb{R}^n}} \right) $ and $ v \circ {u^{ - 1}} \in {C_{[\tau+1] ,{\varpi _2}}}\left( {{\mathbb{R}^n},G} \right) $.
\end{theorem}
\begin{remark}
Similar to \cref{simichaos2}, one can also consider the generalized Logarithmic H\"older's type \eqref{mafan2} instead of the Logarithmic H\"older one. Only the latter is presented here for simplicity.
\end{remark}

\subsection{An explicit example of Logarithmic H\"{o}lder's type}\label{explicitexa}
To illustrate the wider applicability of our theorems, we shall present an explicit example strictly beyond H\"older's type. Note that  the H\"older plus Logarithmic H\"older regularity for $ H $ in \cref{lognew} becomes simpler Logarithmic H\"older's type  for $ 2n<2\tau+2 \in \mathbb{N}^+ $ (because $ \{2 \tau+2 \}=0 $), we therefore consider the following setting.

Recall \cref{lognew}. Let $ n = 2,\tau  = 2, k = 6=[2\tau+2], {\alpha _ * } > 0,\lambda  > 1$ and $M > 0 $ be  given. Assume that $ \left( {x,y} \right) \in {\mathbb{T}^2} \times G $ with $ G := \{ {y \in {\mathbb{R}^2}:\left| y \right| \leq 1} \} $, and the frequency $ \omega  = {\left( {{\omega _1},{\omega _2}} \right)^T} \in \mathbb{R}^2 $ satisfies
\begin{equation}\notag
	| {\langle {{\tilde k},\omega } \rangle } | \geq \alpha_ *{ | {{\tilde k }} |}^{-2} ,\;\;\forall 0 \ne \tilde k \in {\mathbb{Z}^2},\;\;|\tilde k|: = |k_1|+|k_2|,
\end{equation}
i.e., with full Lebesgue measure. Now we shall construct a function for finite smooth  perturbation, whose regularity is $ C^6 $ plus Logarithmic H\"older's type $ \varpi_{\mathrm{LH}}^{\lambda}(r) \sim 1/(-\ln r)^{\lambda} $ with index $ \lambda>1 $. Namely, define
\begin{equation}\notag
	P(r): =  \left\{ \begin{aligned}
		&{{\int_0^r { \cdots \int_0^{{s_2}} {\frac{1}{{{{(1 - \ln \left| {{s_1}} \right|)}^\lambda }}}d{s_1} \cdots d{s_6}} }}},&{0 < \left| r \right| \leq 1} \hfill, \\
		&{0},&{r=0} \hfill. \\
	\end{aligned}  \right.
\end{equation}
Obviously $ P(r)\in C_{6,\varpi_{\mathrm{LH}}^{\lambda}} ([-1,1])$.
Let us consider the perturbed Hamiltonian function below with some constant $ 0 < {\varepsilon } < {\varepsilon ^ * } $ sufficiently small ($ {\varepsilon ^ * } $ depends on the constants given above):
\begin{equation}\label{HH}
	H(x,y,\varepsilon ) = {\omega _1}{y_1} + {\omega _2}{y_2} + \frac{1}{M}(y_1^2 + y_2^2) + \varepsilon \left( {\sin (2\pi {x_1}) + \sin (2\pi {x_2}) + P({y_1}) + P\left( {{y_2}} \right)} \right).
\end{equation}
At this point, we have
\begin{align*}
	\left| {{{\left( {\int_{{\mathbb{T}^2}} {{H_{yy}}\left( {\xi ,0} \right)d\xi } } \right)}^{ - 1}}} \right| &= \left| {{{\left( {\int_{{\mathbb{T}^2}} {\left( {\begin{array}{*{20}{c}}
								{2{M^{ - 1}}}&0 \\
								0&{2{M^{ - 1}}}
						\end{array}} \right)d\xi } } \right)}^{ - 1}}} \right| \notag \\
	&= \left| {\left( {\begin{array}{*{20}{c}}
				{{2^{ - 1}}M}&0 \\
				0&{{2^{ - 1}}M}
		\end{array}} \right)} \right| \leq M <  + \infty .
\end{align*}
In addition, one can verify that $ H \in {C_{6,{\varpi_{\mathrm{LH}}^{\lambda}}}}( {{\mathbb{T}^2} \times G} ) $ with $ \varpi_{\mathrm{LH}}^{\lambda}(r) \sim 1/(-\ln r)^{\lambda} $.

However, for $ \tilde \alpha  = {\left( {0,0,6,0} \right)^T} $ with $ \left| {\tilde \alpha } \right| = 6 = k $, we have
\[\left| {{\partial ^{\tilde \alpha }}H\left( {{{\left( {0,0} \right)}^T},{{\left( {{y_1},0} \right)}^T},\varepsilon } \right) - {\partial ^{\tilde \alpha }}H\left( {{{\left( {0,0} \right)}^T},{{\left( {0,0} \right)}^T},\varepsilon } \right)} \right| = \frac{\varepsilon }{{{{(1 - \ln \left| {{y_1}} \right|)}^\lambda }}} \geq \varepsilon {c_{\lambda ,\ell }}{\left| {{y_1}} \right|^\ell }\]
for any $ 0<\ell\leq1 $, where $ c_{\lambda ,\ell } >0$ is a constant that only depends on $ \lambda $ and $ \ell $. This implies that $ H \notin {C_{6,{\varpi _{\mathrm{H}}^\ell}}}( {{\mathbb{T}^2} \times G} )  $ with $ {\varpi _{\mathrm{H}}^\ell}(r) \sim {r^\ell } $, i.e., $ H \notin C^{6+\ell}( {{\mathbb{T}^2} \times G} ) $ with any $ 0<\ell \leq 1 $, because $ \varpi_{\mathrm{LH}}^{\lambda} $ is strictly weaker than $ \varpi _{\mathrm{H}}^\ell $, see also \cref{strict}.

In other words, the highest derivatives (of order $ k=6 $) of $ H $ in \eqref{HH} can be rigorously proved to be Logarithmic H\"{o}lder continuous with index $ \lambda>1 $, but not any  H\"{o}lder's type. Therefore, the finite smooth KAM theorems via classical H\"{o}lder continuity cannot be applied. But, all the assumptions of  \cref{lognew} can be verified to be satisfied, then the invariant torus persists, and the frequency $ \omega  = {\left( {{\omega _1},{\omega _2}} \right)^T} $ for the unperturbed system can remain unchanged. Moreover, the remaining regularity for mappings $ u $ and $ v \circ u^{-1} $ in \cref{lognew} could also be determined as $ u \in {C_{1 ,{\varpi _\mathrm{LH}^{\lambda-1}}}}\left( {{\mathbb{R}^n},{\mathbb{R}^n}} \right) $   and $ v \circ {u^{ - 1}} \in {C_{3 ,{\varpi _\mathrm{LH}^{\lambda-1}}}}\left( {{\mathbb{R}^n},G} \right) $, where $ \varpi _\mathrm{LH}^{\lambda-1}(r)\sim 1/(-\ln r)^{\lambda-1} $. More precisely,  $ u $ is at least $ C^1 $, while $ v \circ u^{-1}$ is least $ C^3 $, and the higher regularity for them is still not any H\"older's type, but Logarithmic H\"older one with index $ \lambda-1 $, i.e., lower than the original index $ \lambda>1 $, this is because the small divisors causes the loss of regularity.

\section{Proof of  \cref{theorem1}}\label{section4}

Now let us prove  \cref{theorem1} by separating two subsections, namely frequency-preserving KAM persistence (\cref{KAM}) and further regularity (\cref{furtherregularity}) for KAM torus. For the former, the overall process is similar to that in \cite{salamon}, but the key points to weaken the H\"older regularity to only modulus of continuity are using \cref{Theorem1} and proving the uniform convergence of the transformation mapping, that is, the convergence of the upper bound series (see \eqref{dao} and \eqref{dao2}). As we will see later, the Dini type integrability condition \eqref{Dini} in \textbf{(H1)} guarantees this. As to the latter, we have to establish a more general regularity iterative theorem (\cref{t1}) which is not trivial since the resulting regularity might be somewhat  complicated due to asymptotic analysis.

\subsection{Frequency-preserving KAM persistence}\label{KAM}
The proof of the frequency-preserving KAM persistence is organized as follows. Firstly, we construct a series of analytic approximation functions $ H^\nu $ of $ H $ by using  \cref{Theorem1} and considering \textbf{(H1)} and \textbf{(H2)}. Secondly, we shall construct a sequence of frequency-preserving analytic and symplectic transformations $ \psi^\nu $ by induction. According to \textbf{(H2)}, \textbf{(H3)} and \textbf{(H4)}, the first step of induction is established by applying  \cref{appendix} in \cref{Appsalamon} (or Theorem 1 in \cite{salamon}). Then, combining with weak homogeneity and certain specific estimates we complete the proof of induction and obtain the uniform convergence of the composite transformations. Finally, in the light of \textbf{(H5)}, the regularity of the KAM torus is guaranteed by  \cref{t1}.
\\
{\textbf{Step1:}} In view of   \cref{Theorem1} (we have assumed that the modulus of continuity $ \varpi $ admits semi separability and thus \cref{Theorem1} could be applied here), one could approximate $ H(x, y) $ by a sequence of real analytic functions $ H_\nu(x, y) $ for $ \nu \geq 0 $ in the strips
\[\left| {\operatorname{Im} x} \right| \leq {r_\nu },\;\;\left| {\operatorname{Im} y} \right| \leq {r_\nu },\;\;{r_\nu }: = {2^{ - \nu }}\varepsilon \]
around $ \left| {\operatorname{Re} x} \right| \in {\mathbb{T}^n},\left| {\operatorname{Re} y} \right| \leq \rho , $ such that
\begin{equation}\label{T1-3}
	\begin{aligned}
		\left| {{H^\nu }\left( z \right) - \sum\limits_{\left| \alpha  \right| \leq k} {{\partial ^\alpha }H\left( {\operatorname{Re} z} \right)\frac{{{{\left( {\mathrm{i}\operatorname{Im} z} \right)}^\alpha }}}{{\alpha !}}} } \right| \leq{}& {c_1}{\left\| H \right\|_\varpi }r_\nu ^k\varpi \left( {{r_\nu }} \right),\\
		\left| {H_y^\nu \left( z \right) - \sum\limits_{\left| \alpha  \right| \leq k-1} {{\partial ^\alpha }{H_y}\left( {\operatorname{Re} z} \right)\frac{{{{\left( {\mathrm{i}\operatorname{Im} z} \right)}^\alpha }}}{{\alpha !}}} } \right| \leq{}& {c_1}{\left\| H \right\|_\varpi }r_\nu ^{k - 1}\varpi \left( {{r_\nu }} \right), \\
		\left| {H_{yy}^\nu \left( z \right) - \sum\limits_{\left| \alpha  \right| \leq k-2} {{\partial ^\alpha }{H_{yy}}\left( {\operatorname{Re} z} \right)\frac{{{{\left( {\mathrm{i}\operatorname{Im} z} \right)}^\alpha }}}{{\alpha !}}} } \right| \leq{}& {c_1}{\left\| H \right\|_\varpi }r_\nu ^{k - 2}\varpi \left( {{r_\nu }} \right)
	\end{aligned}
\end{equation}
for $ \left| {\operatorname{Im} x} \right| \leq {r_\nu },\;\left| {\operatorname{Im} y} \right| \leq {r_\nu } $, and $ c_1=c(n,k) $ is the constant provided in \eqref{3.2}.

Fix $ \theta  = 1/\sqrt 2  $. In what follows,  we will construct a sequence of real analytic symplectic transformations $ z=(x,y),\zeta=(\xi,\eta),z = {\phi ^\nu }\left( \zeta  \right) $ of the form
\begin{equation}\label{bhxs}
	x = {u^\nu }\left( \xi  \right),\;\;y = v^{\nu}\left( \xi  \right) + (u_\xi ^\nu)^T{\left( \xi  \right)^{ - 1}}\eta
\end{equation}
by induction, such that $ {u^\nu }\left( \xi  \right) - \xi  $ and $ {v^\nu }\left( \xi  \right) $ are of period $ 1 $ in all variables, and $ {\phi ^\nu } $ maps the strip $ \left| {\operatorname{Im} \xi } \right| ,\left| \eta  \right| \leq \theta {r_{\nu  + 1}} $ into $ \left| {\operatorname{Im} x} \right| ,\left| y \right| \leq {r_\nu },\left| {\operatorname{Re} y} \right| \leq \rho  $, and the transformed Hamiltonian function $ 	{K^\nu }: = {H^\nu } \circ {\phi ^\nu } $ satisfies
\begin{equation}\label{qiudao}
	K_\xi ^\nu \left( {\xi ,0} \right) = 0,\;\;K_\eta ^\nu \left( {\xi ,0} \right) = \omega ,
\end{equation}
i.e., with prescribed frequency-preserving. Namely by verifying certain conditions we obtain $ z=\psi^\nu(\zeta) $ of the form \eqref{bhxs} from \cref{appendix} by induction, mapping $ \left| {\operatorname{Im} \xi } \right|,\left| \eta  \right| \leq {r_{\nu  + 1}} $ into $ \left| {\operatorname{Im} x} \right|,\left| y \right| \leq \theta {r_\nu } $, and $ \psi^\nu \left( {\xi ,0} \right) - \left( {\xi ,0} \right) $ is of period $ 1 $, and \eqref{qiudao} holds. Here we denote $ \phi^\nu:=\phi^{\nu-1} \circ \psi^\nu$ with $ {\phi ^{ - 1}}: = \mathrm{id} $ (where $ \mathrm{id} $ denotes the $ 2n $-dimensional identity mapping and therefore $ {\phi ^0} = {\psi ^0} $). Further more, \cref{appendix} will lead to
\begin{align}
	\left| {{\psi ^\nu }\left( \zeta  \right) - \zeta } \right| &\leq c\left( {1 - \theta } \right)r_\nu ^{k - 2\tau  - 1}\varpi \left( {{r_\nu }} \right),\label{2.82}\\
	\left| {\psi _\zeta ^\nu\left( \zeta  \right) - \mathbb{I}} \right| &\leq cr_\nu ^{k - 2\tau  - 2}\varpi \left( {{r_\nu }} \right),\label{2.83}\\
	\left| {K_{\eta \eta }^\nu\left( \zeta  \right) - {Q^\nu}\left( \zeta  \right)} \right| &\leq cr_\nu ^{k - 2\tau  - 2}\varpi \left( {{r_\nu }} \right)/2M,\label{2.84}\\
	\left| {U_x^\nu \left( x \right)} \right| &\leq cr_\nu ^{k - \tau  - 1}\varpi \left( {{r_\nu }} \right),\label{2.85}
\end{align}
on $ \left| {\operatorname{Im} \xi } \right|,\left| \eta  \right|,\left| {\operatorname{Im} x} \right| \leq r_{\nu+1} $, where $ {S^\nu }\left( {x,\eta } \right) = {U^\nu }\left( x \right) + \left\langle {{V^\nu }\left( x \right),\eta } \right\rangle  $ is the generating function for $ {\psi ^\nu } $, and $ Q^\nu:=K_{\eta \eta}^{\nu-1} $, and $ \mathbb{I} $ denotes the $ 2n \times 2n $-dimensional identity mapping,  and
\begin{equation}\label{Q0}
	{Q^0}\left( z \right): = \sum\limits_{\left| \alpha  \right| \leq k - 2} {{\partial ^\alpha }{H_{yy}}\left( {\operatorname{Re} z} \right)\frac{{{{\left( {\mathrm{i}\operatorname{Im} x} \right)}^\alpha }}}{{\alpha !}}} .
\end{equation}
\\
{\textbf{Step2:}} Here we show that $ \psi^0=\phi^0 $ exists, and it admits the properties mentioned in Step 1. Denote
\begin{equation}\notag
	h(x) := H\left( {x,0} \right) - \int_{{\mathbb{T}^n}} {H\left( {\xi ,0} \right)d\xi } ,\;\;x \in {\mathbb{R}^n}.
\end{equation}
Then by the first term in \eqref{T1-2}, we have
\begin{equation}\label{pianh}
	\sum\limits_{\left| \alpha  \right| \leq k} {\left| {{\partial ^\alpha }h} \right|{\varepsilon ^{\left| \alpha  \right|}}}  < M{\varepsilon ^k}\varpi \left( \varepsilon  \right).
\end{equation}
Note that
\begin{align*}
	{H^0}\left( {x,0} \right) - \int_{{\mathbb{T}^n}} {{H^0}\left( {\xi ,0} \right)d\xi }  ={}& {H^0}\left( {x,0} \right) - \sum\limits_{\left| \alpha  \right| \leq k} {\partial _x^\alpha H\left( {\operatorname{Re} x,0} \right)\frac{{{{\left( {\mathrm{i}\operatorname{Im} x} \right)}^\alpha }}}{{\alpha !}}} \notag \\
	{}&+ \int_{{\mathbb{T}^n}} {\left( {H\left( {\xi ,0} \right) - {H^0}\left( {\xi ,0} \right)} \right)d\xi } \notag \\
	{}&+ \sum\limits_{\left| \alpha  \right| \leq k} {{\partial ^\alpha }h\left( {\operatorname{Re} x} \right)\frac{{{{\left( {\mathrm{i}\operatorname{Im} x} \right)}^\alpha }}}{{\alpha !}}} .
\end{align*}
Hence, for $ \left| {\operatorname{Im} x} \right| \leq \theta {r_0} = \theta \varepsilon  $, by using \cref{Theorem1}, \cref{coro1} and \eqref{pianh} we arrive at
\begin{align*}
	\left| {{H^0}\left( {x,0} \right) - \int_{{\mathbb{T}^n}} {{H^0}\left( {\xi ,0} \right)d\xi } } \right| &\leq 2{c_1}{\left\| H \right\|_\varpi }{\varepsilon ^k}\varpi \left( \varepsilon  \right) + M{\varepsilon ^k}\varpi \left( \varepsilon  \right)\notag \\
	&\leq c{\varepsilon ^k}\varpi \left( \varepsilon  \right) \leq c{\varepsilon ^{k - 2\tau  - 2}}\varpi \left( \varepsilon  \right) \cdot {\left( {\theta \varepsilon } \right)^{2\tau  + 2}}.
\end{align*}
Now consider the vector valued function $ 	f\left( x \right): = {H_y}\left( {x,0} \right) - \omega  $ for $ x \in {\mathbb{R}^n} $. In view of the second term in \eqref{T1-2}, we have
\begin{equation}\label{pianf}
	\sum\limits_{\left| \alpha  \right| \leq k - 1} {\left| {{\partial ^\alpha }f} \right|{\varepsilon ^{\left| \alpha  \right|}}}  \leq M{\varepsilon ^{k - \tau  - 1}}\varpi \left( \varepsilon  \right).
\end{equation}
Note that
\begin{align*}
	H_y^0\left( {x,0} \right) - \omega  ={}& H_y^0\left( {x,0} \right) - \sum\limits_{\left| \alpha  \right| \leq k - 1} {\partial _x^\alpha {H_y}\left( {\operatorname{Re} x,0} \right)\frac{{{{\left( {\mathrm{i}\operatorname{Im} x} \right)}^\alpha }}}{{\alpha !}}} \notag \\
	{}&+ \sum\limits_{\left| \alpha  \right| \leq k - 1} {{\partial ^\alpha }f\left( {\operatorname{Re} x} \right)\frac{{{{\left( {\mathrm{i}\operatorname{Im} x} \right)}^\alpha }}}{{\alpha !}}}.
\end{align*}
Therefore, for $ \left| {\operatorname{Im} x} \right| \leq \theta \varepsilon  $, by using \eqref{T1-3} and \eqref{pianf} we obtain that
\begin{align*}
	\left| {H_y^0\left( {x,0} \right) - \omega } \right| &\leq {c_1}{\left\| H \right\|_\varpi }{\varepsilon ^{k - 1}}\varpi \left( \varepsilon  \right) + M{\varepsilon ^{k - \tau  - 1}}\varpi \left( \varepsilon  \right)\notag \\
	&\leq c{\varepsilon ^{k - \tau  - 1}}\varpi \left( \varepsilon  \right) \leq c{\varepsilon ^{k - 2\tau  - 2}}\varpi \left( \varepsilon  \right) \cdot {\left( {\theta \varepsilon } \right)^{\tau  + 1}}.
\end{align*}
Recall \eqref{Q0}. Then it follows from \eqref{T1-3} that
\begin{align*}
	\left| {H_{yy}^0\left( z \right) - {Q^0}\left( z \right)} \right| &\leq {c_1}{\left\| H \right\|_\varpi }{\varepsilon ^{k - 2}}\varpi \left( \varepsilon  \right) \leq \frac{c}{{4M}}{\varepsilon ^{k - 2}}\varpi \left( \varepsilon  \right) \\
	&\leq \frac{c}{{4M}}{\varepsilon ^{k - 2\tau  - 2}}\varpi \left( \varepsilon  \right),\;\;\left| {\operatorname{Im} x} \right|,\left| y \right| \leq \theta \varepsilon,
\end{align*}
and
\begin{equation}\notag
	\left| {{Q^0}\left( z \right)} \right| \leq \sum\limits_{\left| \alpha  \right| \leq k - 2} {{{\left\| H \right\|}_\varpi }\frac{{{\varepsilon ^{\left| \alpha  \right|}}}}{{\alpha !}}}  \leq {\left\| H \right\|_\varpi }\sum\limits_{\alpha  \in {\mathbb{N}^{2n}}} {\frac{{{\varepsilon ^{\left| \alpha  \right|}}}}{{\alpha !}}}  = {\left\| H \right\|_\varpi }{e^{2n\varepsilon }} \leq 2M,\;\; \left| {\operatorname{Im} z} \right| \leq \varepsilon  .
\end{equation}

Now, by taking $ r^{*} = \theta \varepsilon ,\delta^{*}  = {\varepsilon ^{k - 2\tau  - 2}}\varpi \left( \varepsilon  \right) $ and using  \cref{appendix} there exists a real analytic symplectic transformation $ z = {\phi ^0}\left( \zeta  \right) $ of the form \eqref{bhxs} (with $ \nu=0 $) mapping the strip $ \left| {\operatorname{Im} \xi } \right|,\left| \eta  \right| \leq  {r_1}=r_0/2 $ into $ \left| {\operatorname{Im} x} \right|,\left| y \right| \leq \theta{r_0}=r_0/\sqrt{2} $,
such that $ {u^0}\left( \xi  \right) - \xi  $ and $ {v^0}\left( \xi  \right) $ are of period $ 1 $ in all variables and the Hamiltonian function $ {K^0}: = {H^0} \circ {\phi ^0} $  satisfies \eqref{qiudao} (with $ \nu=0 $). Moreover, \eqref{2.82}-\eqref{2.84}  (with $ \nu=0 $) hold.

Also assume that
\begin{equation}\notag
	\left| {K_{\eta \eta }^{\nu  - 1}\left( \zeta  \right)} \right| \leq {M_{\nu  - 1}},\;\;\left| {{{\left( {\int_{{\mathbb{T}^n}} {K_{\eta \eta }^{\nu  - 1}\left( {\xi ,0} \right)d\xi } } \right)}^{ - 1}}} \right| \leq {M_{\nu  - 1}},\;\;{M_\nu } \leq M
\end{equation}
for $ \left| {\operatorname{Im} x} \right| ,\left| y \right| \leq {r_\nu } $. Finally, define
\[\tilde H\left( {x,y} \right): = {H^\nu } \circ {\phi ^{\nu  - 1}}\left( {x,y} \right)\]
with respect to $ \left| {\operatorname{Im} x} \right| ,\left| y \right| \leq {r_\nu } $. One can verify that $ {\tilde H} $ is well defined.

Next we assume that the transformation $ z = {\phi ^{\nu  - 1}}\left( \zeta  \right) $ of the form \eqref{bhxs} has been constructed, mapping $ \left| {\operatorname{Im} \xi } \right|,\left| \eta  \right| \leq \theta {r_\nu } $ into $ \left| {\operatorname{Im} x} \right|,\left| {\operatorname{Im} y} \right| \leq {r_{\nu  - 1}},\left| {\operatorname{Re} y} \right| \leq \rho  $, and $ {u^{\nu  - 1}}\left( \xi  \right) - \xi ,{v^{\nu  - 1}}\left( \xi  \right) $ are of period $ 1 $ in all variables, and $ K_\xi ^{\nu  - 1}\left( {\xi ,0} \right) = 0,K_\eta ^{\nu  - 1}\left( {\xi ,0} \right) = \omega  $. In addition, we also assume that \eqref{2.82}-\eqref{2.85} hold for $ 0, \ldots ,\nu  - 1 $. In the next Step 3, we will verify that the above still hold for $ \nu $, which establishes a complete induction.
\\
{\textbf{Step3:}} We will prove the existence of transformation $ {\phi ^\nu } $ in each step according to the specific estimates below and  \cref{appendix}.

Let $ \left| {\operatorname{Im} x} \right| \leq \theta {r_\nu } $. Then $ \phi^{\nu-1}(x,0) $ lies in the region where the estimates in \eqref{T1-3} hold for both $ H^\nu $ and $ H^{\nu-1} $. Note that $ x \mapsto H^{\nu-1}(\phi^{\nu-1}(x,0)) $ is constant by \eqref{qiudao}. Then by \eqref{T1-3}, we arrive at the following for $ \left| {\operatorname{Im} x} \right| \leq \theta {r_\nu } $
\begin{align*}
	\left| {\tilde H\left( {x,0} \right) - \int_{{\mathbb{T}^n}} {\tilde H\left( {\xi ,0} \right)d\xi } } \right| &\leq 2\mathop {\sup }\limits_{\left| {\operatorname{Im} \xi } \right| \leq \theta {r_\nu }} \left| {{H^\nu }\left( {{\phi ^{\nu  - 1}}\left( {\xi ,0} \right)} \right) - {H^{\nu  - 1}}\left( {{\phi ^{\nu  - 1}}\left( {\xi ,0} \right)} \right)} \right|\notag \\
	&\leq 2{c_1}{\left\| H \right\|_\varpi }r_\nu ^k\varpi \left( {{r_\nu }} \right) + 2{c_1}{\left\| H \right\|_\varpi }r_{\nu-1} ^{k}\varpi \left( {{r_{\nu-1} }} \right)\notag \\
	&\leq cr_\nu ^{k - 2\tau  - 2}\varpi \left( {{r_\nu }} \right) \cdot r_\nu ^{2\tau  + 2},
\end{align*}
where the weak homogeneity of $ \varpi $ with respect to $ a=1/2 $ (see \cref{weak}) has been used in the last inequality, because $ \varpi(r_{\nu-1})=\varpi(2r_{\nu})\leq c \varpi(r_{\nu}) $ (thus $ c $ is independent of $ \nu $). For convenience we may therefore not mention it in the following.

Taking $ \eta=0 $ in \eqref{2.83} we have
\begin{align}
	\left| {u_\xi ^{\nu  - 1}\left( \xi  \right) - \mathbb{I}} \right| &\leq \sum\limits_{\mu  = 0}^{\nu  - 1} {\left| {u_\xi ^\mu \left( \xi  \right) - u_\xi ^{\mu  - 1}\left( \xi  \right)} \right|}  \leq c\sum\limits_{\mu  = 0}^{\nu  - 1} {r_\mu ^{k - 2\tau  - 2}\varpi \left( {{r_\mu }} \right)}  \notag \\
	&\leq c\sum\limits_{\mu  = 0}^\infty  {{{\left( {\frac{\varepsilon }{{{2^\mu }}}} \right)}^{k - 2\tau  - 2}}\varpi \left( {\frac{\varepsilon }{{{2^\mu }}}} \right)}  \leq c\sum\limits_{\mu  = 0}^\infty  {\left( {\frac{\varepsilon }{{{2^{\mu  - 1}}}} - \frac{\varepsilon }{{{2^\mu }}}} \right){{\left( {\frac{\varepsilon }{{{2^\mu }}}} \right)}^{k - 2\tau  - 3}}\varpi \left( {\frac{\varepsilon }{{{2^\mu }}}} \right)}  \notag \\
\label{dao}	&   \leq c\sum\limits_{\mu  = 0}^\infty  {\int_{\varepsilon /{2^\mu }}^{\varepsilon /{2^{\mu  - 1}}} {\frac{{\varpi \left( x \right)}}{{{x^{2\tau  + 3 - k}}}}dx} }  \leq c\int_0^{2\varepsilon } {\frac{{\varpi \left( x \right)}}{{{x^{2\tau  + 3 - k}}}}dx}  \leq 1 - \theta
\end{align}
for $ \left| {\operatorname{Im} \xi } \right| \leq \theta {r_\nu } $, and the Dini type condition \eqref{Dini} in \textbf{(H1)} together with Cauchy Theorem are used since $ \varepsilon>0 $ is sufficiently small. Then it leads to
\begin{equation}\label{nidao}
	\left| {u_\xi ^{\nu  - 1}{{\left( \xi  \right)}^{ - 1}}} \right| \leq {\theta ^{ - 1}},\;\;\left| {\operatorname{Im} \xi } \right| \leq \theta {r_\nu }.
\end{equation}

Finally, by \eqref{nidao} and \eqref{T1-3} we obtain that
\begin{align*}
	\left| {{{\tilde H}_y}\left( {x,0} \right) - \omega } \right| &= \left| {u_\xi ^{\nu  - 1}{{\left( x \right)}^{ - 1}}\left( {H_y^\nu \left( {{\phi ^{\nu  - 1}}\left( {x,0} \right)} \right) - H_y^{\nu  - 1}\left( {{\phi ^{\nu  - 1}}\left( {x,0} \right)} \right)} \right)} \right|\notag \\
	& \leq {\theta ^{ - 1}}\left| {H_y^\nu \left( {{\phi ^{\nu  - 1}}\left( {x,0} \right)} \right) - H_y^{\nu  - 1}\left( {{\phi ^{\nu  - 1}}\left( {x,0} \right)} \right)} \right|\notag \\
	& \leq {\theta ^{ - 1}}\left( {{c_1}{{\left\| H \right\|}_\varpi }r_\nu ^{k - 1}\varpi \left( {{r_\nu }} \right) + {c_1}{{\left\| H \right\|}_\varpi }r_{\nu  - 1}^{k - 1}\varpi \left( {{r_{\nu  - 1}}} \right)} \right)\notag \\
	&\leq cr_\nu ^{k - 1}\varpi \left( {{r_\nu }} \right)\notag \\
	&\leq cr_\nu ^{k - \tau  - 2}\varpi \left( {{r_\nu }} \right) \cdot r_\nu ^{\tau  + 1},
\end{align*}
and
\begin{align*}
	\left| {{{\tilde H}_{yy}}\left( z \right) - {Q^\nu }\left( z \right)} \right| &= \left| {u_\xi ^{\nu  - 1}{{\left( x \right)}^{ - 1}}\left( {H_{yy}^\nu \left( {{\phi ^{\nu  - 1}}\left( z \right)} \right) - H_{yy}^{\nu  - 1}\left( {{\phi ^{\nu  - 1}}\left( z \right)} \right)} \right){{\left( {u_\xi ^{\nu  - 1}{{\left( x \right)}^{ - 1}}} \right)}^T}} \right|\notag \\
	&\leq {\theta ^{ - 2}}\left| {H_{yy}^\nu \left( {{\phi ^{\nu  - 1}}\left( z \right)} \right) - H_{yy}^{\nu  - 1}\left( {{\phi ^{\nu  - 1}}\left( z \right)} \right)} \right|\notag \\
	&\leq {\theta ^{ - 2}}\left( {{c_1}{{\left\| H \right\|}_\varpi }r_\nu ^{k - 2}\varpi \left( {{r_\nu }} \right) + {c_1}{{\left\| H \right\|}_\varpi }r_{\nu  - 1}^{k - 2}\varpi \left( {{r_{\nu  - 1}}} \right)} \right)\notag \\
	&\leq cr_\nu ^{k - 2\tau  - 2}\varpi \left( {{r_\nu }} \right)/2M
\end{align*}
for $ \left| {\operatorname{Im} x} \right|,\left| y \right| \leq \theta {r_\nu } $. Then denote $ r^*:= r_\nu $ and $ \delta^*:=c r_\nu ^{k - 2\tau  - 2}\varpi \left( {{r_\nu }} \right) $ in   \cref{appendix}, we obtain the analytic symplectic preserving transformation $ {\phi ^\nu } $ of each step, mapping the strip $ \left| {\operatorname{Im} \xi } \right|\leq \theta {r_\nu },\left| \eta  \right| \leq \theta {r_\nu } $ into $ \left| {\operatorname{Im} x} \right|\leq {r_\nu },\left| y \right| \leq {r_\nu } $, such that $ {u^\nu }\left( \xi  \right) - \xi $ and $ {v^\nu }\left( \xi  \right) $ are of period $ 1 $ in all variables, and the transformed Hamiltonian function $ {K^\nu } = {H^\nu } \circ {\phi ^\nu } $ satisfies
\[K_\xi ^\nu \left( {\xi ,0} \right) = 0,\;\;K_\eta ^\nu \left( {\xi ,0} \right) = \omega .\]
Moreover, \eqref{2.82}-\eqref{2.85} are valid for $ \left| {\operatorname{Im} \xi } \right|,\left| \eta  \right|,\left| {\operatorname{Im} x} \right|\leq \theta {r_\nu }$.
\\
{\textbf{Step4:}} By  \eqref{2.83} for $ 0, \ldots ,\nu  - 1 $ and the arguments in \eqref{dao}, there holds
\begin{align}
	\left| {\phi _\zeta ^{\nu  - 1}\left( \zeta  \right)} \right| &\leq 1 + \sum\limits_{\mu  = 0}^{\nu  - 1} {\left| {\phi _\zeta ^\mu \left( \zeta  \right) - \phi _\zeta ^{\mu  - 1}\left( \zeta  \right)} \right|}  \leq 1 + \sum\limits_{\mu  = 0}^{\nu  - 1} {\left( {\left| {\phi _\zeta ^\mu \left( \zeta  \right) - \mathbb{I}} \right| + \left| {\phi _\zeta ^{\mu  - 1}\left( \zeta  \right) - \mathbb{I}} \right|} \right)} \notag \\
\label{dao2}	&\leq 1 + c\sum\limits_{\mu  = 0}^\infty  {{{\left( {\frac{\varepsilon }{{{2^\mu }}}} \right)}^{k - 2\tau  - 2}}\varpi \left( {\frac{\varepsilon }{{{2^\mu }}}} \right)}   \leq 1 + c\int_0^{2\varepsilon } {\frac{{\varpi \left( x \right)}}{x^{2\tau +3-k}}dx}   \leq 2
\end{align}
for $ \left| {\operatorname{Im} \xi } \right|,\left| \eta  \right| \leq \theta {r_\nu } $ as long as $ \varepsilon>0 $ is sufficiently small, which leads to
\begin{align*}
	\left| {{\phi ^\nu }\left( \zeta  \right) - {\phi ^{\nu  - 1}}\left( \zeta  \right)} \right| &= \left| {{\phi ^{\nu  - 1}}\left( {{\psi ^\nu }\left( \zeta  \right)} \right) - {\phi ^{\nu  - 1}}\left( \zeta  \right)} \right| \notag \\
	&\leq 2\left| {{\psi ^\nu }\left( \zeta  \right) - \zeta } \right| \leq c\left( {1 - \theta } \right)r_\nu ^{k - 2\tau  - 1}\varpi \left( {{r_\nu }} \right)
\end{align*}
for $ \left| {\operatorname{Im} \xi } \right|,\left| \eta  \right| \leq {r_{\nu  + 1}} $. Then by Cauchy's estimate, we obtain that
\begin{equation}\notag
	\left| {\phi _\zeta ^\nu \left( \zeta  \right) - \phi _\zeta ^{\nu  - 1}\left( \zeta  \right)} \right| \leq cr_\nu ^{k - 2\tau  - 2}\varpi \left( {{r_\nu }} \right),\;\;\left| {\operatorname{Im} \xi } \right|,\left| \eta  \right| \leq {r_{\nu  + 1}}.
\end{equation}
It can be proved in the same way that $ | {\phi _\zeta ^\nu \left( \zeta  \right)} | \leq 2 $ for $ \left| {\operatorname{Im} \xi } \right|,\left| \eta  \right| \leq \theta {r_{\nu  + 1}} $, which implies
\begin{equation}\notag
	\left| {\operatorname{Im} z} \right| \leq 2\left| {\operatorname{Im} \zeta } \right| \leq 2\sqrt {{{\left| {\operatorname{Im} \xi } \right|}^2} + {{\left| {\operatorname{Im} \eta } \right|}^2}}  \leq 2\sqrt {{\theta ^2}r_{\nu  + 1}^2 + {\theta ^2}r_{\nu  + 1}^2}  = 2{r_{\nu  + 1}} = {r_\nu }.
\end{equation}
Besides, we have $ \left| {\operatorname{Re} y} \right| \leq \rho  $.

Note that
\begin{equation}\notag
	{v^\nu } \circ {\left( {{u^\nu }} \right)^{ - 1}}\left( x \right) - {v^{\nu  - 1}} \circ {\left( {{u^{\nu  - 1}}} \right)^{ - 1}}\left( x \right) = {\left( {u_\xi ^{\nu  - 1}{{\left( \xi  \right)}^{ - 1}}} \right)^T}U_x^\nu \left( \xi  \right),\;\;x: = {u^{\nu  - 1}}\left( \xi  \right).
\end{equation}
Recall \eqref{dao}, by employing the contraction mapping principle we have $ \left| {\operatorname{Im} \xi } \right| \leq {r_{\nu  + 1}} $ if $ \left| {\operatorname{Im} x} \right| \leq \theta {r_{\nu  + 1}} $ with respect to $ x $ defined above. Then from \eqref{2.85} and \eqref{nidao} one can verify that
\begin{equation}\label{4.97}
	\left| {{{\big( {u_\xi ^{\nu  - 1}{{\left( \xi  \right)}^{ - 1}}} \big)}^T}U_x^\nu \left( \xi  \right)} \right| \leq cr_\nu ^{k - \tau  - 1}\varpi \left( {{r_\nu }} \right).
\end{equation}
{\textbf{Step5:}} Finally, we are in a position  to  prove the convergence of $ u^\nu $ and $ v^\nu $, and the regularity of their limit functions. Note \eqref{4.97}. Then we have the following analytic iterative scheme
\begin{equation}\label{4.98}
	\left| {{u^\nu }\left( \xi  \right) - {u^{\nu  - 1}}\left( \xi  \right)} \right| \leq cr_\nu ^{k - 2\tau  - 1}\varpi \left( {{r_\nu }} \right), \;\; \left| {\operatorname{Im} \xi } \right| \leq {r_{\nu  + 1}},
\end{equation}
and
\begin{equation}\label{4.99}
	\left| {{v^\nu } \circ {{\left( {{u^\nu }} \right)}^{ - 1}}\left( x \right) - {v^{\nu  - 1}} \circ {{\left( {{u^{\nu  - 1}}} \right)}^{ - 1}}\left( x \right)} \right| \leq c r_\nu ^{k - \tau  - 1}\varpi \left( {{r_\nu }} \right),\;\;\left| {\operatorname{Im} x} \right| \leq \theta {r_{\nu  + 1}}.
\end{equation}
And especially, \eqref{4.98} and \eqref{4.99} hold when $ \nu=0 $ since $ {u^{0 - 1}} = \mathrm{id} $ and $ {v^{ 0- 1}} = 0 $. It is obvious to see that the uniform limits $ u $ and $ v\circ u^{-1} $  of  $ u^\nu $ and $ v^\nu\circ (u^\nu)^{-1} $ are at least $ C^1 $ (in fact, this is implied by the higher regularity studied later in \cref{furtherregularity}). In addition, the persistent invariant torus possesses the same frequency $ \omega $ as the unperturbed torus by \eqref{qiudao}.

\subsection{Iteration theorem on regularity without H\"older's type}\label{furtherregularity}
To obtain accurate regularity for $ u $ and $ v\circ u^{-1} $ from the analytic iterative scheme \eqref{4.98} and \eqref{4.99},  we shall along with the idea of Moser and Salamon to establish an abstract iterative theorem, which provides the  modulus of continuity of the integral form.
\begin{theorem}\label{t1}
	Let $ n\in \mathbb{N}^+, \varepsilon>0 $ and $ \{r_\nu\}_{\nu \in \mathbb
	N}=\{\varepsilon2^{-\nu}\}_{n \in \mathbb{N}} $ be given, and denote by $ f:{\mathbb{R}^n} \to \mathbb{R} $ the limit of a sequence of real analytic functions $ {f_\nu }\left( x \right) $ in the strips $ \left| {\operatorname{Im} x} \right| \leq {r_{\nu} } $ such that
	\begin{equation}\label{huoche}
		{f_0} = 0,\;\;\left| {{f_\nu }\left( x \right) - {f_{\nu  - 1}}\left( x \right)} \right| \leq \varphi \left( {{r_\nu }} \right),\;\;\nu  \geq 1,
	\end{equation}
	where $ \varphi $ is a nondecreasing continuous function satisfying $  \varphi \left( 0 \right) = 0  $.		Assume that there is a critical $ k_* \in \mathbb{N} $ such that
	\begin{equation}\label{330}
	\int_0^1 {\frac{{\varphi \left( x \right)}}{{{x^{{k_*} + 1}}}}dx}  <  + \infty ,\;\;\int_0^1 {\frac{{\varphi \left( x \right)}}{{{x^{{k_*} + 2}}}}dx}  =  + \infty .
	\end{equation}
	Then there exists a modulus of continuity $ \varpi_*  $  such that $ f \in {C_{k_*,\varpi_* }}\left( {{\mathbb{R}^n}} \right) $.  In other words, the regularity of $ f $ is at least of $ C^{k_*} $ plus $ \varpi_* $. In particular, $ \varpi_* $ could be determined as
	\begin{equation}\label{LLL}
	{\varpi _ * }\left( \gamma \right) \sim	\gamma \int_{L\left( \gamma  \right)}^\varepsilon  {\frac{{\varphi \left( t \right)}}{{{t^{{k_*} + 2}}}}dt}  = {\mathcal{O}^\# }\left( {\int_0^{L\left( \gamma  \right)} {\frac{{\varphi \left( t \right)}}{{{t^{{k_*} + 1}}}}dt} } \right) ,\;\;\gamma  \to {0^ + },
	\end{equation}
	where $ L(\gamma) \to 0^+ $ is some function such that the second relation in \eqref{LLL} holds.
\end{theorem}
\begin{proof}
Define $ {g_\nu }(x): = {f_\nu }\left( x \right) - {f_{\nu  - 1}}\left( x \right) $ for $ \nu \in \mathbb{N}^+ $.	Determine an integer function $ \widetilde N(\gamma) : [0,1] \to \mathbb{N}^+ $ (note that $ \widetilde N(\gamma) $ can be extended to $ \mathbb{R}^+ $ due to the arguments below, we thus assume that  it is  a continuous function). Then for the given critical  $ k_*\in \mathbb{N} $ and $ x,y\in \mathbb{R}^n $, we obtain  the following  for all multi-indices $ \alpha  = \left( {{\alpha _1}, \ldots ,{\alpha _n}} \right) \in {\mathbb{N}^n} $ with $ \left| \alpha  \right| = k_* $:
\begin{align}
\sum\limits_{\nu  = 1}^{\widetilde N\left( {\left| {x - y} \right|} \right)-1} {\left| {{\partial ^\alpha }{g_\nu }\left( x \right) - {\partial ^\alpha }{g_\nu }\left( y \right)} \right|}&\leq \left| {x - y} \right|\sum\limits_{\nu  = 1}^{\widetilde N\left( {\left| {x - y} \right|} \right)-1} {{{\left| {{\partial ^\alpha }{g_{\nu x}}} \right|}_{{C^0}(\mathbb{R}^n)}}}\leq \left| {x - y} \right|\sum\limits_{\nu  = 1}^{\widetilde N\left( {\left| {x - y} \right|} \right)-1} {\frac{1}{{r_\nu ^{k_* + 1}}}\varphi \left( {{r_\nu }} \right)} \notag \\
& = 2\left| {x - y} \right|\sum\limits_{\nu  = 1}^{\widetilde N\left( {\left| {x - y} \right|} \right)-1} {\left( {\frac{\varepsilon }{{{2^\nu }}} - \frac{\varepsilon }{{{2^{\nu  + 1}}}}} \right){{\left( {\frac{{{2^\nu }}}{\varepsilon }} \right)}^{{k_ * } + 2}}\varphi \left( {\frac{\varepsilon }{{{2^\nu }}}} \right)}\notag  \\
\label{zhengze1}&\leq  c\left| {x - y} \right|\int_{\varepsilon {2^{ - \widetilde N\left( {\left| {x - y} \right|} \right) }}}^\varepsilon  {\frac{{\varphi \left( t \right)}}{{{t^{{k_*} + 2}}}}dt} ,
\end{align}
where Cauchy's estimate and \eqref{huoche} are used in the second inequality, and arguments similar  to  \eqref{dao} are employed in \eqref{zhengze1}, $ c>0 $ is a universal constant. Besides, we similarly get
\begin{align}
\sum\limits_{\nu  = \widetilde N\left( {\left| {x - y} \right|} \right) }^\infty  {\left| {{\partial ^\alpha }{g_\nu }\left( x \right) - {\partial ^\alpha }{g_\nu }\left( y \right)} \right|} &\leq \sum\limits_{\nu  = \widetilde N\left( {\left| {x - y} \right|} \right) }^\infty  {2{{\left| {{\partial ^\alpha }{g_\nu }} \right|}_{{C^0}(\mathbb{R}^n)}}}\leq 2\sum\limits_{\nu  = \widetilde N\left( {\left| {x - y} \right|} \right)}^\infty  {\frac{1}{{r_\nu ^{k_*}}}\varphi \left( {{r_\nu }} \right)}\notag \\
&=2\sum\limits_{\nu  = \widetilde N\left( {\left| {x - y} \right|} \right)}^\infty  {\left( {\frac{\varepsilon }{{{2^\nu }}} - \frac{\varepsilon }{{{2^{\nu  + 1}}}}} \right){{\left( {\frac{{{2^\nu }}}{\varepsilon }} \right)}^{{k_ * } + 1}}\varphi \left( {\frac{\varepsilon }{{{2^\nu }}}} \right)}\notag \\
\label{zhengze2}& \leq c\int_0^{\varepsilon {2^{ - \widetilde N\left( {\left| {x - y} \right|} \right) }}} {\frac{{\varphi \left(t \right)}}{{{t^{{k_*} + 1}}}}dt} .
\end{align}
Now choose $ \widetilde{N}(\gamma) \to +\infty $ as $ \gamma \to 0^+ $ such that
\begin{equation}\label{varpi*}
\gamma \int_{L\left( \gamma  \right)}^\varepsilon  {\frac{{\varphi \left( t \right)}}{{{t^{{k_*} + 2}}}}dt}  = {\mathcal{O}^\# }\left( {\int_0^{L\left( \gamma  \right)} {\frac{{\varphi \left( t \right)}}{{{t^{{k_*} + 1}}}}dt} } \right): = {\varpi _ * }\left( \gamma \right),\;\;\gamma  \to {0^ + },
\end{equation}
where $ \varepsilon {2^{ - \tilde N\left( \gamma  \right) - 1}}: = L\left( \gamma  \right)\to 0^+ $. This is achievable due to assumption \eqref{330}, Cauchy Theorem  and The Intermediate Value Theorem. Note that the choice of $ L(\gamma) $ (i.e., $ \widetilde{N} $) and $ \varpi_* $ is not unique (may up to a constant), and $ \varpi_* $ could be continuously extended to some given interval (e.g., $ [0,1] $), but this does not affect the qualitative result. Combining \eqref{zhengze1}, \eqref{zhengze2} and \eqref{varpi*} we finally arrive at $ f \in {C_{k_*,\varpi_* }}\left( {{\mathbb{R}^n}} \right) $ because
 \begin{align*}
 \left| {{\partial ^\alpha }f\left( x \right) - {\partial ^\alpha }f\left( y \right)} \right| &\leq \sum\limits_{\nu  = 1}^{\widetilde N\left( {\left| {x - y} \right|} \right)} { + \sum\limits_{\nu  = \widetilde N\left( {\left| {x - y} \right|} \right) + 1}^\infty  {\left| {{\partial ^\alpha }{g_\nu }\left( x \right) - {\partial ^\alpha }{g_\nu }\left( y \right)} \right|} } \\
 & \leq c\left( {\left| {x - y} \right|\int_{\varepsilon {2^{ - \widetilde N\left( {\left| {x - y} \right|} \right) - 1}}}^\varepsilon  {\frac{{\varphi \left( t \right)}}{{{t^{{k_*} + 2}}}}dt}  + \int_0^{\varepsilon {2^{ - \widetilde N\left( {\left| {x - y} \right|} \right) - 1}}} {\frac{{\varphi \left( t \right)}}{{{t^{{k_*} + 1}}}}dt} } \right) \\
 &\leq c\varpi_* \left( {\left| {x - y} \right|} \right).
 \end{align*}

\end{proof}

\cref{t1} can be extended to the case $ f:{\mathbb{R}^n} \to \mathbb{R}^m $ with $ n,m \in \mathbb{N}^+ $ since the analysis is completely the same,  and the strip $ \left| {\operatorname{Im} x} \right| \leq {r_\nu } $ can also be replaced by $ \left| {\operatorname{Im} x} \right| \leq {r_{\nu+1} } $ (or $ \leq \theta r_{\nu+1} $).  \cref{t1} can also be used to estimate the regularity of solutions of finite smooth homological equations, thus KAM uniqueness theorems in some cases might be derived, see Section 4 in \cite{salamon} for instance. However, in order to avoid too much content in this paper, it is omitted here.

Recall \eqref{4.98} and \eqref{4.99}. Then one can apply  \cref{t1} on $ \{u^{\nu}-\mathrm{id}\}_\nu $ (because \cref{t1} requires that the initial value vanishes) and $ \{v^{\nu}\circ (u^\nu)^{-1}\}_\nu $ to directly analyze the regularity of the KAM torus  according to \textbf{(H5)}, i.e., there exist $ {\varpi _i} $ ($ i=1,2 $) such that $ u \in {C_{k_1^ * ,{\varpi _1}}}\left( {{\mathbb{R}^n},{\mathbb{R}^n}} \right) $ and $ v \circ {u^{ - 1}} \in {C_{k_2^ * ,{\varpi _2}}}\left( {{\mathbb{R}^n},G} \right) $. This completes the proof of \cref{theorem1}.

\section{Proof of \cref{simichaos2}}\label{proofsimichaos2}
	Only need to determine $ k_i^* $ in \textbf{(H5)} and choose functions $ L_i(\gamma)\to 0^+ $ (as $ \gamma \to 0^+ $) to obtain the modulus of continuity $ \varpi_i $ in \eqref{varpii} for $ i=1,2 $. Obviously $ k_1^*=1 $ and $ k_2^*=n $ because
	\[\int_0^1 {\frac{{\varpi _{\mathrm{GLH}}^{\varrho ,\lambda }\left( x \right)}}{x}dx}  <  + \infty ,\;\;\int_0^1 {\frac{{\varpi _{\mathrm{GLH}}^{\varrho ,\lambda }\left( x \right)}}{{{x^2}}}dx}  =  + \infty .\]
	In view of $ \varphi_i(x) $  in \textbf{(H5)}, then by applying \cref{duochongduishu} we get
	\begin{align}
		\gamma \int_{{L_i}\left( \gamma  \right)}^\varepsilon  {\frac{{{\varphi _i}\left( t \right)}}{{{t^{k_i^ *  + 2}}}}dt}  &= {\mathcal{O}^\# }\Bigg( {\gamma \int_{{L_i}\left( \gamma  \right)}^\varepsilon  {\frac{1}{{{t^2}(\ln (1/t)) \cdots {{(\underbrace {\ln  \cdots \ln }_\varrho (1/t))}^\lambda }}}dt} } \Bigg)\notag \\
		& = {\mathcal{O}^\# }\Bigg( {\gamma \int_{1/\varepsilon }^{1/{L_i}\left( \gamma  \right)} {\frac{1}{{(\ln z) \cdots {{(\underbrace {\ln  \cdots \ln }_\varrho z)}^\lambda }}}dz} } \Bigg)\notag \\
		\label{guji1}& = {\mathcal{O}^\# }\Bigg( {\frac{\gamma }{{{L_i}\left( \gamma  \right)(\ln (1/{L_i}\left( \gamma  \right)) \cdots {{(\underbrace {\ln  \cdots \ln }_\varrho (1/{L_i}\left( \gamma  \right)))}^\lambda }}}} \Bigg),
	\end{align}
	and by direct calculation one arrives at
	\begin{align}
		\int_0^{{L_i}\left( \gamma  \right)} {\frac{{{\varphi _i}\left( t \right)}}{{{t^{k_i^ *  + 1}}}}dt} & = {\mathcal{O}^\# }\Bigg( {\int_{{L_i}\left( \gamma  \right)}^\varepsilon  {\frac{1}{{t(\ln (1/t)) \cdots {{(\underbrace {\ln  \cdots \ln }_\varrho (1/t))}^\lambda }}}dt} } \Bigg)\notag \\
		\label{guji2}&= {\mathcal{O}^\# }\Bigg( {\frac{1}{{{{(\underbrace {\ln  \cdots \ln }_\varrho (1/{L_i}\left( \gamma  \right)))}^{\lambda  - 1}}}}} \Bigg).
	\end{align}
	Finally, choosing
	\begin{equation}\label{Lit4}
		{L_i}\left( \gamma  \right) \sim \frac{{\gamma }}{{(\ln (1/\gamma)) \cdots (\underbrace {\ln  \cdots \ln }_{\varrho}(1/\gamma )})} \to {0^ + },\;\;\gamma  \to {0^ + }
	\end{equation}
	will lead to the second relation in \eqref{varpii} for $ i=1,2 $, and substituting $ L_i(\gamma) $ into \eqref{guji1} or \eqref{guji2} yields that
\begin{equation}\label{rho1}
	{\varpi _1}\left( \gamma \right) \sim {\varpi _2}\left( \gamma \right) \sim \frac{1}{{{{(\underbrace {\ln  \cdots \ln }_\varrho (1/\gamma))}^{\lambda  - 1}}}}.
\end{equation}
	in \cref{theorem1}, see \eqref{varpii}. This proves \cref{simichaos2}.

\section{Proof of  \cref{Holder}}\label{8}
Note that $ \ell  \notin {\mathbb{N}^ + } $ implies $ \{\ell\}\in (0,1) $. Then $ k=[\ell] $ and $ \varpi(x)\sim \varpi_{\mathrm{H}}^{\ell}(x)\sim x^{\{\ell\}} $, i.e., modulus of continuity of H\"older's type. Consequently, \textbf{(H1)} can be directly verified because of  $ \ell  > 2\tau  + 2 $:
\[\int_0^1 {\frac{{\varpi \left( x \right)}}{{{x^{2\tau  + 3 - k}}}}dx}  = \int_0^1 {\frac{{{x^{\left\{ \ell  \right\}}}}}{{{x^{2\tau  + 3 - \left[ \ell  \right]}}}}dx}  = \int_0^1 {\frac{1}{{{x^{1 - \left( {\ell  - 2\tau  - 2} \right)}}}}dx}  <  + \infty .\]
Here and below, let $ i $ be $ 1 $ or $ 2 $ for simplicity. Recall that $ {\varphi _i}\left( x \right) = {x^{k - \left( {3 - i} \right)\tau  - 1}}\varpi \left( x \right) = {x^{\left[ \ell  \right] - \left( {3 - i} \right)\tau  - 1}} \cdot {x^{\left\{ \ell  \right\}}} = {x^{\ell  - \left( {3 - i} \right)\tau  - 1}} $, and let
\begin{align}
\label{fff}\int_0^1 {\frac{{{\varphi _i}\left( x \right)}}{{{x^{k_i^ *  + 1}}}}dx}  &= \int_0^1 {\frac{1}{{{x^{k_i^ *  - \left( {\ell  - \left( {3 - i} \right)\tau  - 2} \right)}}}}dx}<+\infty,\\
\label{ffff}\int_0^1 {\frac{{{\varphi _i}\left( x \right)}}{{{x^{k_i^ *  + 2}}}}dx}  &= \int_0^1 {\frac{1}{{{x^{k_i^ *  - \left( {\ell  - \left( {3 - i} \right)\tau  - 2} \right) + 1}}}}dx}=+\infty .
\end{align}
Then the critical $ k_i^* $ in \textbf{(H5)} could be uniquely chosen as $ k_i^ * : = \left[ {\ell  - \left( {3 - i} \right)\tau  - 1} \right] \in \mathbb{N}^+$ since $ \ell  - \left( {3 - i} \right)\tau  - 1 \notin \mathbb{N}^+ $. Further, letting $ {L_i}\left( \gamma  \right) = \gamma  \to {0^ + } $ yields that
\[\int_0^{{L_i}\left( \gamma  \right)} {\frac{{{\varphi _i}\left( t \right)}}{{{t^{k_i^ *  + 1}}}}dt}  = {\mathcal{O}^\# }\left( {\int_0^\gamma  {\frac{1}{{{t^{1 - \left\{ {\ell  - \left( {3 - i} \right)\tau  - 2} \right\}}}}}dt} } \right) = {\mathcal{O}^\# }\left( {{\gamma ^{\left\{ {\ell  - \left( {3 - i} \right)\tau  - 2} \right\}}}} \right)\]
and
\[\gamma \int_{{L_i}\left( \gamma  \right)}^\varepsilon  {\frac{{{\varphi _i}\left( t \right)}}{{{t^{k_i^ *  + 2}}}}dt}  = {\mathcal{O}^\# }\left( {\gamma \int_\gamma ^\varepsilon  {\frac{1}{{{t^{2 - \left\{ {\ell  - \left( {3 - i} \right)\tau  - 2} \right\}}}}}dt} } \right) = {\mathcal{O}^\# }\left( {{\gamma ^{\left\{ {\ell  - \left( {3 - i} \right)\tau  - 2} \right\}}}} \right).\]
This leads to H\"older's type
\[{\varpi _i}\left( \gamma  \right) \sim {\left( {{L_i}\left( \gamma  \right)} \right)^{\left\{ {\ell  - \left( {3 - i} \right)\tau  - 2} \right\}}} \sim {\gamma ^{\left\{ {\ell  - \left( {3 - i} \right)\tau  - 2} \right\}}}\sim \varpi_{\mathrm{H}}^{\left\{ {\ell  - \left( {3 - i} \right)\tau  - 2} \right\}}(\gamma)\]
due to \eqref{varpii} in \cref{theorem1}. By observing $ k_i^ *  + \left\{ {\ell  - \left( {3 - i} \right)\tau  - 2} \right\} = \ell  - \left( {3 - i} \right)\tau  - 1 $ we finally arrive at $ u \in {C^{\ell-2\tau-1}}\left( {{\mathbb{R}^n},{\mathbb{R}^n}} \right) $ and $ v \circ {u^{ - 1}} \in {C^{\ell-\tau-1}}\left( {{\mathbb{R}^n},G} \right) $. This proves \cref{Holder}.

\section{Proof of \cref{lognew}}\label{pflognew}
Firstly, note that $ k=[2\tau+2] $ and $ \varpi \left( x \right) \sim {x^{\{2\tau+2\}}}/{\left( { - \ln x} \right)^\lambda } $ with $ \lambda  > 1 $, then \textbf{(H1)} holds since
\begin{align*}
\int_0^1 {\frac{{\varpi \left( x \right)}}{{{x^{2\tau  + 3 - k}}}}dx}  &= {\mathcal{O}^\# }\left( {\int_0^{1/2} {\frac{{{x^{\left\{ {2\tau  + 2} \right\}}}}}{{{x^{2\tau  + 3 - \left[ {2\tau  + 2} \right]}}{{\left( { - \ln x} \right)}^\lambda }}}dx} } \right)\\
& = {\mathcal{O}^\# }\left( {\int_0^{1/2} {\frac{1}{{x{{\left( { - \ln x} \right)}^\lambda }}}dx} } \right) <  + \infty .
\end{align*}
Secondly, in view of $ \varphi_i(x) $ in \textbf{(H5)}, we have
\[\int_0^1 {\frac{{{\varphi _i}\left( x \right)}}{{{x^{k_i^ *  + 1}}}}dx}  = {\mathcal{O}^\# }\left( {\int_0^{1/2} {\frac{1}{{{x^{k_i^ *  - \left( {i - 1} \right)\tau }}{{\left( { - \ln x} \right)}^\lambda }}}dx} } \right)\;\;i=1,2.\]
This leads to critical $ k_1^*=1 $ and $ k_2^*=[\tau+1] $ in \textbf{(H5)}. Here one uses the following fact: for given $ \lambda>1 $,
\[\int_0^{1/2} {\frac{1}{{{x^\iota }{{\left( { - \ln x} \right)}^\lambda }}}dx}  <  + \infty ,\;\;\int_0^{1/2} {\frac{1}{{{x^{\iota  + 1}}{{\left( { - \ln x} \right)}^\lambda }}}dx}=+\infty \]
if and only if $ \iota \in (0,1] $.

Next, we investigate the KAM remaining regularity through certain complicated asymptotic analysis. One notices that the analysis of $ \varpi_1 $ with all $ \tau>n-1 $ and $ \varpi_2 $ with $ n-1<\tau \notin \mathbb{N}^+ $ is the same as $ \varrho =1$ in \cref{simichaos2}, i.e., $ L_1(\gamma) $ and $ L_2(\gamma) $ could be chosen as $ \gamma/(-\ln \gamma) \to 0^+$, see \eqref{Lit4} with $ \varrho =1$.
Therefore, in view of \eqref{rho1}, we arrive at
\[{\varpi _1}\left( \gamma \right) \sim \frac{1}{{{{\left( { - \ln \gamma} \right)}^{\lambda  - 1}}}} \sim \varpi _{\mathrm{LH}}^{\lambda  - 1}\left( \gamma \right), \;\; \tau >n-1,\]
and
\[{\varpi _2}\left( \gamma \right) \sim \frac{1}{{{{\left( { - \ln \gamma} \right)}^{\lambda  - 1}}}} \sim \varpi _{\mathrm{LH}}^{\lambda  - 1}\left( \gamma \right), \;\; n-1<\tau \notin \mathbb{N}^+.\]
However, the asymptotic analysis for $ \varpi_2 $ becomes different when $ n-1<\tau \in \mathbb{N}^+  $. Note that $ \{\tau\} \in (0,1) $ and $ \left[ {\tau  + 1} \right] - \tau  = \left[ \tau  \right] + 1 - \tau  = 1 - \left\{ \tau  \right\} $ at present. Hence, by  applying \eqref{erheyi1} in \cref{erheyi} we get
\begin{align}
\int_0^{{L_2}\left( \gamma  \right)} {\frac{{{\varphi _2}\left( t \right)}}{{{t^{k_2^ *  + 1}}}}dt}  &= {\mathcal{O}^\# }\left( {\int_0^{{L_2}\left( \gamma  \right)} {\frac{1}{{{t^{\left[ {\tau  + 1} \right] - \tau }}{{\left( { - \ln t} \right)}^\lambda }}}dt} } \right) = {\mathcal{O}^\# }\left( {\int_0^{{L_2}\left( \gamma  \right)} {\frac{1}{{{t^{1 - \left\{ \tau  \right\}}}{{\left( { - \ln t} \right)}^\lambda }}}dt} } \right)\notag \\
\label{coro42}& = {\mathcal{O}^\# }\left( {\int_{1/{L_2}\left( \gamma  \right)}^{ + \infty } {\frac{1}{{{z^{1 + \left\{ \tau  \right\}}}{{\left( {\ln z} \right)}^\lambda }}}dz} } \right) = {\mathcal{O}^\# }\left( {\frac{{{{\left( {{L_2}\left( \gamma  \right)} \right)}^{\left\{ \tau  \right\}}}}}{{{{\left( {\ln \left( {1/{L_2}\left( \gamma  \right)} \right)} \right)}^\lambda }}}} \right),
\end{align}
and similarly according to \eqref{erheyi2} in \cref{erheyi} we have
\begin{equation}\label{coro43}
\gamma \int_{{L_2}\left( \gamma  \right)}^\varepsilon  {\frac{{{\varphi _2}\left( t \right)}}{{{t^{k_2^ *  + 2}}}}dt}  = {\mathcal{O}^\# }\left( {\gamma \int_{1/\varepsilon }^{1/{L_2}\left( \gamma  \right)} {\frac{1}{{{z^{\left\{ \tau  \right\}}}{{\left( {\ln z} \right)}^\lambda }}}dz} } \right) = {\mathcal{O}^\# }\left( {\frac{{\gamma {{\left( {{L_2}\left( \gamma  \right)} \right)}^{\left\{ \tau  \right\} - 1}}}}{{{{\left( {\ln \left( {1/{L_2}\left( \gamma  \right)} \right)} \right)}^\lambda }}}} \right).	
\end{equation}
Now let us choose $ L_2(\gamma) \sim \gamma \to 0^+ $, i.e., different from that when $ n-1<\tau \in \mathbb{N}^+ $, 	one verifies that the second relation in \eqref{varpii} holds for $ i=2 $, and substituting $ L_2(\gamma) $ into \eqref{coro42} or \eqref{coro43} yields that
\[{\varpi _2}\left( \gamma  \right) \sim \frac{{{\gamma ^{\left\{ \tau  \right\}}}}}{{{{\left( { - \ln \gamma } \right)}^\lambda }}} \sim {\gamma ^{\left\{ \tau  \right\}}}\varpi _{\mathrm{LH}}^\lambda \left( \gamma  \right),\;\; n-1<\tau \notin \mathbb{N}^+\]
due to \eqref{varpii} in \cref{theorem1}. This proves \cref{lognew}.

\appendix
\section{Semi separability and weak homogeneity for modulus of continuity}
\begin{lemma}\label{Oxlemma}
Let a modulus continuity $ \varpi $ be given. If $ \varpi $ is piecewise continuously differentiable and $ \varpi'\geq 0 $ is nonincreasing, then $ \varpi $ admits semi separability in \cref{d2}. As a consequence, if $ \varpi $ is convex near $ 0^+ $, then it is semi separable.
\end{lemma}
\begin{proof}
Assume that $ \varpi $ is continuously differentiable without loss of generality. Then we obtain semi separability due to
\begin{align*}
\mathop {\sup }\limits_{0 < r < \delta /x} \frac{{\varpi \left( {rx} \right)}}{{\varpi \left( r \right)}} &= \mathop {\sup }\limits_{0 < r < \delta /x} \frac{{\varpi \left( {rx} \right) - \varpi \left( 0+ \right)}}{{\varpi \left( r \right)}} \leq \mathop {\sup }\limits_{0 < r < \delta /x} \frac{1}{{\varpi \left( r \right)}}\sum\limits_{j = 0}^{\left[ x \right]} {\int_{jr}^{\left( {j + 1} \right)r} {\varpi '\left( t \right)dt} } \\
& \leq \mathop {\sup }\limits_{0 < r < \delta /x} \frac{1}{{\varpi \left( r \right)}}\sum\limits_{j = 0}^{\left[ x \right]} {\int_0^r {\varpi '\left( t \right)dt} }  = \left( {\left[ x \right] + 1} \right) = \mathcal{O}\left( x \right),\;\;x \to  + \infty .
\end{align*}
\end{proof}
\begin{lemma}\label{ruotux}
Let a modulus continuity $ \varpi $ be given. If $ \varpi $ is convex near $ 0^+ $, then it admits weak homogeneity in \cref{weak}.
\end{lemma}
\begin{proof}
For $ x>0 $ sufficiently small, one verifies that
\[\varpi \left( x \right) = x \cdot \frac{{\varpi \left( x \right) - \varpi \left( 0+ \right)}}{x-0} \leq x \cdot \frac{{\varpi \left( {ax} \right) - \varpi \left( 0+ \right)}}{{ax}-0} = a^{-1}\varpi \left( {ax} \right),\]
for  $ 0<a<1 $, which leads to weak homogeneity
\[	\mathop {\overline {\lim } }\limits_{x \to {0^ + }} \frac{{\varpi \left( x \right)}}{{\varpi \left( {ax} \right)}} \leq a^{-1} <  + \infty .\]
\end{proof}

\section{Proof of \cref{Theorem1}} \label{JACK}
\begin{proof}
For the completeness of the analysis we give a very detailed proof. An outline of the strategy for the proof is provided: we firstly construct an approximation integral operator by the Fourier transform of  a compactly supported function, and then present certain  properties of the operator (note that these preparations are classical); finally, we estimate the approximation error in the sense of modulus of continuity.	
	
	Let\[K\left( x \right) = \frac{1}{{{{\left( {2\pi } \right)}^n}}}\int_{{\mathbb{R}^n}} {\widehat K\left( \xi \right){e^{\mathrm{i}\left\langle {x,\xi } \right\rangle }}d\xi } ,\;\;x \in {\mathbb{C}^n}\]
	be an entire function whose Fourier transform
	\[\widehat K\left( \xi  \right) = \int_{{\mathbb{R}^n}} {K\left( x \right){e^{ - \mathrm{i}\left\langle {x,\xi } \right\rangle }}dx} ,\;\;\xi  \in {\mathbb{R}^n}\]
	is a smooth function with compact support, contained in the ball $ \left| \xi  \right| \leq 1 $, that satisfies $ \widehat K\left( \xi  \right) = \widehat K\left( { - \xi } \right) $ and
	\begin{equation}\label{3.3}
		{\partial ^\alpha }\widehat K\left( 0 \right) = \left\{ \begin{gathered}
			1,\;\;\alpha  = 0, \hfill \\
			0,\;\;\alpha  \ne 0. \hfill \\
		\end{gathered}  \right.
	\end{equation}
	Next, we assert that
	\begin{equation}\label{5}
		\left| {{\partial ^\beta }\mathcal{F}\big( {\widehat K\left( \xi  \right)} \big)\left( z \right)} \right|  \leq \frac{{c\left( {\beta ,p} \right)}}{{{{\left( {1 + \left| {\operatorname{Re} z} \right|} \right)}^p}}}{e^{\left| {\operatorname{Im} z} \right|}},\;\;\max \left\{ {1,\left|\beta\right|} \right\} \leq p \in \mathbb{R}.
	\end{equation}
	
	Note that we assume $ \widehat K \in C_0^\infty \left( {{\mathbb{R}^n}} \right) $ and $ \operatorname{supp} \widehat K \subseteq B\left( {0,1} \right) $, thus
	\begin{equation}\label{FK}
		\left| {{{\left( {1 + \left| z \right|} \right)}^k}{\partial ^\beta }\mathcal{F}\left( {\widehat K\left( \xi  \right)} \right)\left( z \right)} \right| \leq \sum\limits_{\left| \gamma  \right| \leq k} {\left| {{z^\gamma }{\partial ^\beta }\mathcal{F}\big( {\widehat K\left( \xi  \right)} \big)\left( z \right)} \right|}  = \sum\limits_{\left| \gamma  \right| \leq k} {\left| {{\partial ^{\beta  + \gamma }}\mathcal{F}\big( {\widehat K\left( \xi  \right)} \big)\left( z \right)} \right|},
	\end{equation}
	where $ \mathcal{F} $ represents the Fourier transform. Since $ {\partial ^{\beta  + \gamma }}\mathcal{F}\big( {\widehat K\left( \xi  \right)} \big)\left( z \right) \in C_0^\infty ( {\overline {B \left( {0,1} \right)}} ) $ does not change the condition, we only need to prove that
	\[\left| {\mathcal{F}\big( {\widehat K\left( \xi  \right)} \big)\left( z \right)} \right| \leq {c_k}{e^{\left| {\operatorname{Im} z} \right|}}.\]
	Obviously
	\[\left| {\mathcal{F}\big( {\widehat K\left( \xi  \right)} \big)\left( z \right)} \right| \leq \frac{1}{{{{\left( {2\pi } \right)}^n}}}\int_{{\mathbb{R}^n}} {\big| {\widehat K\left( \xi  \right)} \big|{e^{ - \left\langle {\operatorname{Im} z,\xi } \right\rangle }}d\xi }  \leq \frac{c}{{{{\left( {2\pi } \right)}^n}}}\int_{B\left( {0,1} \right)} {{e^{\left| {\left\langle {\operatorname{Im} z,\xi } \right\rangle } \right|}}d\xi }  \leq c{e^{\left| {\operatorname{Im} z} \right|}},\]
	where $ c>0 $ is independent of $ n $. Then assertion \eqref{5} is proved by recalling  \eqref{FK}.
	
	The inequality in \eqref{5} is usually called the Paley-Wiener Theorem, see also Chapter III in \cite{Stein}.	As we will see later, it plays an important role in the subsequent verification of definitions, integration by parts and the translational feasibility according to Cauchy's integral formula.
	
	Next we assert that $ K:{\mathbb{C}^n} \to \mathbb{R} $ is a real analytic function with the following property
	\begin{equation}\label{3.6}
		\int_{{\mathbb{R}^n}} {{{\left( {u + \mathrm{i}v} \right)}^\alpha }{\partial ^\beta }K\left( {u + \mathrm{i}v} \right)du}  = \left\{ \begin{aligned}
			&{\left( { - 1} \right)^{\left| \alpha  \right|}}\alpha !,&\alpha  = \beta , \hfill \\
			&0,&\alpha  \ne \beta , \hfill \\
		\end{aligned}  \right.
	\end{equation}
	for $ u,v \in {\mathbb{R}^n} $ and multi-indices $ \alpha ,\beta  \in {\mathbb{N}^n} $. In order to prove assertion  \eqref{3.6}, we first consider proving the following for $ x\in \mathbb{R}^n $:
	\begin{equation}\label{3.7}
		\int_{{\mathbb{R}^n}} {{x^\alpha }{\partial ^\beta }K\left( x \right)dx}  = \left\{ \begin{aligned}
			&{\left( { - 1} \right)^{\left| \alpha  \right|}}\alpha !,&\alpha  = \beta  \hfill, \\
			&0,&\alpha  \ne \beta  \hfill. \\
		\end{aligned}  \right.
	\end{equation}
	{\bf{Case1:}} If $ \alpha  = \beta  $, then
	\begin{align*}
		\int_{{\mathbb{R}^n}} {{x^\alpha }{\partial ^\beta }K\left( x \right)dx}  &= \int_{{\mathbb{R}^n}} {\Big( {\prod\limits_{j = 1}^n {x_j^{{\alpha _j}}} } \Big)\Big( {\prod\limits_{j = 1}^n {\partial _{{x_j}}^{{\alpha _j}}} } \Big)K\left( x \right)dx} \\
		&= {\left( { - 1} \right)^{{\alpha _1}}}{\alpha _1}!\int_{{\mathbb{R}^{n - 1}}} {\Big( {\prod\limits_{j = 2}^n {x_j^{{\alpha _j}}} } \Big)\Big( {\prod\limits_{j = 2}^n {\partial _{{x_j}}^{{\alpha _j}}} } \Big)K\left( x_2,\cdots,x_n \right)dx_2 \cdots dx_n} \\
		&=  \cdots  = {\left( { - 1} \right)^{{\alpha _1} +  \cdots  + {\alpha _n}}}{\alpha _1}! \cdots {\alpha _n}!\int_\mathbb{R} {K\left( x_n \right)dx_n} \\
		& = {\left( { - 1} \right)^{\left| \alpha  \right|}}\alpha !\widehat K\left( 0 \right) = {\left( { - 1} \right)^{\left| \alpha  \right|}}\alpha !.
	\end{align*}
	{\bf{Case2:}} There exists some $ {\alpha _j} \leq {\beta _j} - 1 $, let $ j=1 $ without loss of generality. Then
	\begin{align*}
		\int_{{\mathbb{R}^n}} {{x^\alpha }{\partial ^\beta }K\left( x \right)dx}  &= \int_{{\mathbb{R}^n}} {\Big( {\prod\limits_{j = 1}^n {x_j^{{\alpha _j}}} } \Big)\Big( {\prod\limits_{j = 1}^n {\partial _{{x_j}}^{{\beta _j}}} } \Big)K\left( x \right)dx} \\
		&= {\left( { - 1} \right)^{{\beta _1} - {\alpha _1}}}\int_{{\mathbb{R}^n}} {\Big( {\prod\limits_{j = 2}^n {x_j^{{\alpha _j}}} } \Big)\Big( {\prod\limits_{j = 2}^n {\partial _{{x_j}}^{{\beta _j}}} } \Big)\partial _{{x_1}}^{{\beta _1} - {\alpha _1}}K\left( x \right)dx}=0.
	\end{align*}
	{\bf{Case3:}} Now we have $ {\alpha _1} \geq {\beta _1} $, and some $ {\alpha _j} \geq {\beta _j} + 1 $ (otherwise $ \alpha  = \beta  $). Let $ j=1 $ without loss of generality. At this time we first prove a conclusion according to \eqref{3.3}. Since
	\[{\partial ^\alpha }\widehat K\left( 0 \right) = {\left( { - \mathrm{i}} \right)^{\left| \alpha  \right|}}\int_{{\mathbb{R}^n}} {{x^\alpha }K\left( x \right)dx}  = 0,\;\;\alpha  \ne 0,\]
	then it follows that
	\begin{displaymath}
		\int_{{\mathbb{R}^n}} {{x^\alpha }K\left( x \right)dx}  = 0,\;\;\alpha  \ne 0.
	\end{displaymath}
	Hence, we arrive at
	\[\int_{{\mathbb{R}^n}} {{x^\alpha }{\partial ^\beta }K\left( x \right)dx}  = {\left( { - 1} \right)^{\sum\limits_{j = 1}^n {\left( {{\beta _j} - {\alpha _j}} \right)} }}\int_{{\mathbb{R}^n}} {\left( {x_1^{{\alpha _1} - {\beta _1}}} \right)\Big( {\prod\limits_{j = 2}^n {x_j^{{\alpha _j} - {\beta _j}}} } \Big)K\left( x \right)dx}  = 0.\]
	This proves \eqref{3.7}. Next, we will consider a complex translation of \eqref{3.7} and prove that
	\begin{equation}\notag
		\int_{{\mathbb{R}^n}} {{{\left( {u + \mathrm{i}v} \right)}^\alpha }{\partial ^\beta }K\left( {u + \mathrm{i}v} \right)du}  = \left\{ \begin{aligned}
			&{\left( { - 1} \right)^{\left| \alpha  \right|}}\alpha !,&\alpha  = \beta,  \hfill \\
			&0,&\alpha  \ne \beta . \hfill \\
		\end{aligned}  \right.
	\end{equation}
	Actually one only needs to pay attention to \eqref{5}, and the proof is completed according to Cauchy's integral formula.
	
	Finally, we only prove that
	\begin{equation}\label{3.10}
		{S_r}{p} = {p},\;\;{p}:{\mathbb{R}^n} \to \mathbb{R}.
	\end{equation}
	In fact, only real polynomials need to be considered
	\[{p} = {x^\alpha } = \prod\limits_{j = 1}^n {x_j^{{\alpha _j}}} \ .\]
	It can be obtained by straight calculation
	\begin{align*}
		{S_r}{p} &= {r^{ - n}}\int_{{\mathbb{R}^n}} {K\left( {\frac{{x - y}}{r}} \right)\prod\limits_{j = 1}^n {y_j^{{\alpha _j}}} dy}  = \int_{{\mathbb{R}^n}} {K\left( z \right)\prod\limits_{j = 1}^n {{{\left( {r{z_j} + {x_j}} \right)}^{{\alpha _j}}}} dz} \\
		& = \Big( {\prod\limits_{j = 1}^n {x_j^{{\alpha _j}}} } \Big)\int_{{\mathbb{R}^n}} {K\left( z \right)dz}  + \sum\limits_\gamma  {{\varphi _\gamma }\left( {r,x} \right)\int_{{\mathbb{R}^n}} {{z^\gamma }K\left( z \right)dz} }  = \prod\limits_{j = 1}^n {x_j^{{\alpha _j}}}  = {p},
	\end{align*}
	where $ {{\varphi _\gamma }\left( {r,x} \right)} $ are coefficients independent of  $ z $.
	
	As to the complex case one only needs to perform complex translation to obtain
	\begin{equation}\notag
		{p}\left( {u;\mathrm{i}v} \right) = {S_r}{p}\left( {u + \mathrm{i}v} \right) = \int_{{\mathbb{R}^n}} {K\left( {\mathrm{i}{r^{ - 1}}v - \eta } \right){p}\left( {u;r\eta } \right)d\eta }.
	\end{equation}

	The above preparations are classic, see also \cite{salamon}. Next we begin to prove the Jackson type approximation theorem via only modulus of continuity. 	We will make use of \eqref{3.10} in case of the Taylor polynomial
	\[{p_k}\left( {x;y} \right): = {P_{f,k}}\left( {x;y} \right) = \sum\limits_{\left| \alpha  \right| \leq k} {\frac{1}{{\alpha !}}{\partial ^\alpha }f\left( x \right){y^\alpha }} \]
	of $ f $ with $ k\in \mathbb{N} $. Note that
	\[\left| {f\left( {x + y} \right) - {p_k}\left( {x;y} \right)} \right| = \Bigg| {\int_0^1 {k{{\left( {1 - t} \right)}^{k - 1}}\sum\limits_{\left| \alpha  \right| = k} {\frac{1}{{\alpha !}}\left( {{\partial ^\alpha }f\left( {x + ty} \right) - {\partial ^\alpha }f\left( x \right)} \right){y^\alpha }dt} } } \Bigg|\]
	for every $ x,y \in {\mathbb{R}^n} $.
	
	Define the following domains to partition $ \mathbb{R}^n $:
	\[{\Omega _1}: = \left\{ {\eta  \in {\mathbb{R}^n}:\left| \eta  \right| < \delta{r^{ - 1}}} \right\},\;\;{\Omega _2}: = \left\{ {\eta  \in {\mathbb{R}^n}:\left| \eta  \right| \geq \delta{r^{ - 1}}} \right\}.\]
	We have to use different estimates in the above two domains, which are abstracted as follows.		If $ 0 < \left| y \right| < \delta $,  we obtain that
	\begin{align}
		\left| {f\left( {x + y} \right) - {p_k}\left( {x;y} \right)} \right| \leq{}& \int_0^1 {k{{\left( {1 - t} \right)}^{k - 1}}\sum\limits_{\left| \alpha  \right| = k} {\frac{1}{{\alpha !}} \cdot {{\left[ {{\partial ^\alpha }f} \right]}_\varpi }\varpi \left( {t\left| y \right|} \right) \cdot \left| {{y^\alpha }} \right|dt} } \notag \\
		\leq{}& c\left( {n,k} \right){\left\| f \right\|_\varpi }\int_0^1 {\varpi \left( {t\left| y \right|} \right)dt}  \cdot \left| {{y^\alpha }} \right|\notag \\
		\leq{}& c\left( {n,k} \right){\left\| f \right\|_\varpi }\varpi \left( {\left| y \right|} \right)\left| {{y^\alpha }} \right|\notag \\
		\label{e1}\leq{}& c\left( {n,k} \right){\left\| f \right\|_\varpi }\varpi \left( {\left| y \right|} \right){\left| y \right|^k}.
	\end{align}
	If $ \left| y \right| \geq \delta $, one easily arrives at
	\begin{align}
		\left| {f\left( {x + y} \right) - {p_k}\left( {x;y} \right)} \right| \leq{}& \int_0^1 {k{{\left( {1 - t} \right)}^{k - 1}}\sum\limits_{\left| \alpha  \right| = k} {\frac{1}{{\alpha !}} \cdot 2{{\left| {{\partial ^\alpha }f} \right|}_{{C^0}}} \cdot \left| {{y^\alpha }} \right|dt} } \notag \\
		\leq{}& c\left( {n,k} \right){\left\| f \right\|_\varpi }\left| {{y^\alpha }} \right|\notag \\
		\label{e2}	\leq{}& c\left( {n,k} \right){\left\| f \right\|_\varpi }{\left| y \right|^k}.
	\end{align}
	The H\"{o}lder inequality has been used in \eqref{e1} and \eqref{e2} with $ k \geq 1,{\alpha _i} \geq 1, \mu_i=k/\alpha_i\geq 1 $ without loss of generality:
	\[\left| {{y^\alpha }} \right| = \prod\limits_{i = 1}^n {{{\left| {{y_i}} \right|}^{{\alpha _i}}}}  \leq \sum\limits_{i = 1}^n {\frac{1}{{{\mu _i}}}{{\left| {{y_i}} \right|}^{{\alpha _i}{\mu _i}}}}  \leq \sum\limits_{i = 1}^n {{{\left| {{y_i}} \right|}^k}}  \leq \sum\limits_{i = 1}^n {{{\left| y \right|}^k}}  = n{\left| y \right|^k}.\]
	
	Now let $ x=u+\mathrm{i}v $ with $ u,v \in {\mathbb{R}^n} $ and $ \left| v \right| \leq r $. Fix $ p = n + k + 2 $, and let $ c = c\left( {n,k} \right) > 0 $ be a universal constant, then it follows that
	\begin{align*}
		\left| {{S_r}f\left( {u + \mathrm{i}v} \right) - {p_k}\left( {u;\mathrm{i}v} \right)} \right| \leq{}& \int_{{\mathbb{R}^n}} {K\left( {\mathrm{i}{r^{ - 1}}v - \eta } \right)\left| {f\left( {u + r\eta } \right) - {p_k}\left( {u;r\eta } \right)} \right|d\eta } \notag \\
		\leq{}& c\int_{{\mathbb{R}^n}} {\frac{{{e^{{r^{ - 1}}v}}}}{{{{\left( {1 + \left| \eta  \right|} \right)}^p}}}\left| {f\left( {u + r\eta } \right) - {p_k}\left( {u;r\eta } \right)} \right|d\eta } \notag \\
		\leq{}& c\int_{{\mathbb{R}^n}} {\frac{1}{{{{\left( {1 + \left| \eta  \right|} \right)}^p}}}\left| {f\left( {u + r\eta } \right) - {p_k}\left( {u;r\eta } \right)} \right|d\eta } \notag \\
		={}& c\int_{{\Omega _1}}  +  \int_{{\Omega _2}} {\frac{1}{{{{\left( {1 + \left| \eta  \right|} \right)}^p}}}\left| {f\left( {u + r\eta } \right) - {p_k}\left( {u;r\eta } \right)} \right|d\eta } \notag \\
		: ={}& c\left( {{I_1} + {I_2}} \right).
	\end{align*}
	As it can be seen later, $ I_1 $ is the main part while $ I_2 $ is the remainder.
	
	Recall \cref{Remarksemi} and \eqref{Ox}. Hence the following holds due to \eqref{e1}
	\begin{align}\label{I1}
		{I_1} ={}& \int_{{\Omega _1}} {\frac{1}{{{{\left( {1 + \left| \eta  \right|} \right)}^p}}}\left| {f\left( {u + r\eta } \right) - {p_k}\left( {u;r\eta } \right)} \right|d\eta }\notag \\
		\leq{}& \int_{\left| \eta  \right| < \delta{r^{ - 1}}} {\frac{1}{{{{\left( {1 + \left| \eta  \right|} \right)}^p}}} \cdot c{{\left\| f \right\|}_\varpi }\varpi \left( {\left| {r\eta } \right|} \right){{\left| {r\eta } \right|}^k}d\eta } \notag \\
		\leq{}& \int_{\left| \eta  \right| < \delta{r^{ - 1}}} {\frac{1}{{{{\left( {1 + \left| \eta  \right|} \right)}^p}}} \cdot c{{\left\| f \right\|}_\varpi }\varpi \left( { {r } } \right) \psi(|\eta|) {{\left| {r\eta } \right|}^k}d\eta } \notag \\
		\leq{}& c{\left\| f \right\|_\varpi }{r^k}{\varpi(r)}\int_0^{\delta{r^{ - 1}}} {\frac{{{w^{k + n }}}}{{{{\left( {1 + w} \right)}^p}}}dw} \notag \\
		\leq{}& c{\left\| f \right\|_\varpi }{r^k}{\varpi(r)}\int_0^{+\infty} {\frac{{{w^{k + n }}}}{{{{\left( {1 + w} \right)}^p}}}dw} \notag \\
		\leq{}& c{\left\| f \right\|_\varpi }{r^k}{\varpi(r)}.
	\end{align}
	In view of \eqref{e2}, we have
	\begin{align}\label{I2}
		{I_2} ={}& \int_{{\Omega _2}} {\frac{1}{{{{\left( {1 + \left| \eta  \right|} \right)}^p}}}\left| {f\left( {u + r\eta } \right) - {p_k}\left( {u;r\eta } \right)} \right|d\eta } \notag \\
		\leq{}& \int_{\left| \eta  \right| \geq \delta{r^{ - 1}}} {\frac{1}{{{{\left( {1 + \left| \eta  \right|} \right)}^p}}} \cdot c{{\left\| f \right\|}_\varpi }{{\left| {r\eta } \right|}^k}d\eta } \notag \\
		\leq{}& c{\left\| f \right\|_\varpi }{r^k}\int_{\delta{r^{ - 1}}}^{ + \infty } {\frac{{{w^{k + n - 1}}}}{{{{\left( {1 + w} \right)}^p}}}dw} \notag \\
		\leq{}& c{\left\| f \right\|_\varpi }{r^k}\int_{\delta{r^{ - 1}}}^{ + \infty } {\frac{1}{{{w^{p - k - n + 1}}}}dw} \notag \\
		\leq{}& c{\left\| f \right\|_\varpi }{r^{k+2}}.
	\end{align}
	By \eqref{I1} and \eqref{I2}, we finally arrive at
	\begin{equation}\notag
		\left| {{S_r}f\left( {u + \mathrm{i}v} \right) - {p_k}\left( {u;\mathrm{i}v} \right)} \right| \leq c{\left\| f \right\|_\varpi }{r^k} {\varpi(r)}
	\end{equation}
	due to $ \mathop {\overline {\lim } }\limits_{r \to {0^ + }} r/{\varpi }\left( r \right) <  + \infty  $ in \cref{d1}.	This proves \cref{Theorem1} for $ |\alpha| = 0 $. As to $ |\alpha| \ne 0 $, the result follows from the fact that $ {S_r} $ commutes with $ {\partial ^\alpha } $. We therefore finish the proof of \cref{Theorem1}.
\end{proof}

\section{Proof of \cref{coro1}}\label{proofcoro1}
\begin{proof}
	Only the analysis of case $ \left| \alpha  \right| = 0 $ is given. In view of  \cref{Theorem1} and \eqref{e1}, we obtain that
	\begin{align}
			\left| {{S_r}f\left( x \right) - f\left( x \right)} \right| \leq{}& \left| {{S_r}f\left( x \right) - {P_{f,k}}\left( {\operatorname{Re} x;\mathrm{i}\operatorname{Im} x} \right)} \right| + \left| {{P_{f,k}}\left( {\operatorname{Re} x;\mathrm{i}\operatorname{Im} x} \right) - f\left( x \right)} \right|\notag \\
		\label{Srf-f}	\leq{}& c_*{\left\| f \right\|_\varpi }{r^k}\varpi(r) ,
		\end{align}
	where the constant $ c_*>0 $ depends on $ n $ and $ k $.
	Further, by \eqref{Srf-f} we have
	\begin{align*}
		\left| {{S_r}f\left( x \right)} \right| \leq{}& \left| {{S_r}f\left( x \right) - f\left( x \right)} \right| + \left| {f\left( x \right)} \right|\notag \\
			\leq{}& c_*{\left\| f \right\|_\varpi }{r^k}\varpi(r) + {\left\| f \right\|_\varpi } \leq {c^ * }{\left\| f \right\|_\varpi },
		\end{align*}
provided a constant $ c^*>0 $ depending on $ n,k $ and $ \varpi $.	This completes the proof.
\end{proof}

\section{Proof of \cref{coro2}}\label{proofcoco2}
\begin{proof}
	It is easy to verify that
	\begin{align*}
			{S_r}f\left( {x + 1} \right) ={}& \frac{1}{{{r^n}}}\int_{{\mathbb{R}^n}} {K\left( {\frac{{x - \left( {y - 1} \right)}}{r}} \right)f\left( y \right)dy} = \frac{1}{{{r^n}}}\int_{{\mathbb{R}^n}} {K\left( {\frac{{x - u}}{r}} \right)f\left( {u + 1} \right)du} \notag \\
			={}& \frac{1}{{{r^n}}}\int_{{\mathbb{R}^n}} {K\left( {\frac{{x - u}}{r}} \right)f\left( u \right)du}  = {S_r}f\left( x \right).
		\end{align*}
	According to Fubini's theorem, we obtain
	\begin{align*}
			\int_{{\mathbb{T}^n}} {{S_r}f\left( x \right)dx}  ={}& \frac{1}{{{r^n}}}\int_{{\mathbb{R}^n}} {\int_{{\mathbb{T}^n}} {K\left( {\frac{{x - y}}{r}} \right)f\left( y \right)dy} }  \notag \\
			={}& \frac{1}{{{r^n}}}\int_{{\mathbb{R}^n}} {K\left( {\frac{m}{r}} \right)\left( {\int_{{\mathbb{T}^n}} {f\left( {x + m} \right)dx} } \right) dm}  = 0.
		\end{align*}
	
	This completes the proof.
\end{proof}

\section{Asymptotic analysis in estimates}
Here we provide some useful asymptotic results, all of which can be proved by L'Hopital's rule or by integration by parts,  thus the proof is omitted here.
\begin{lemma}\label{duochongduishu}
Let $ \varrho \in \mathbb{N}^+ $, $ \lambda>1 $ and some $ M>0 $ sufficiently large be fixed. Then for $ X\to +\infty $, there holds
\[\int_M^X {\frac{1}{{(\ln z) \cdots {{(\underbrace {\ln  \cdots \ln }_\varrho z)}^\lambda }}}dz}  = {\mathcal{O}^\# }\Bigg( {\frac{X}{{(\ln X) \cdots {{(\underbrace {\ln  \cdots \ln }_\varrho X)}^\lambda }}}} \Bigg).\]
\end{lemma}

\begin{lemma}\label{erheyi}
	Let $ 0 <\sigma<1 $, $ \lambda>1 $ and some $ M>0 $ sufficiently large be fixed. Then for $ X\to +\infty $, we have
\begin{equation}\label{erheyi1}
	\int_M^X {\frac{1}{{{z^\sigma }{{\left( {\ln z} \right)}^\lambda }}}dz}  = {\mathcal{O}^\# }\left( {\frac{{{X^{1 - \sigma }}}}{{{{\left( {\ln X} \right)}^\lambda }}}} \right),
\end{equation}
and
\begin{equation}\label{erheyi2}
	\int_X^{ + \infty } {\frac{1}{{{z^{1 + \sigma }}{{\left( {\ln z} \right)}^\lambda }}}dz}  = {\mathcal{O}^\# }\left( {\frac{1}{{{X^\sigma }{{\left( {\ln X} \right)}^\lambda }}}} \right).
\end{equation}
\end{lemma}

\section{KAM theorem for quantitative estimates}\label{Appsalamon}
Here we give a KAM theorem for quantitative estimates, which is used in  \cref{theorem1} in this paper. See Theorem 1 in Salamon's paper \cite{salamon} for case $ \tau >n-1 $; as to $ \tau =n-1 $, the proof is relatively trivial (in fact, just slightly modify Lemma 2 in \cite{salamon}).
\begin{theorem}\label{appendix}
	Let $ n \geq 2, \tau \geq n - 1,  0 < \theta < 1$, and $ M \geq 1 $ be given. Then there are positive constants $ \delta_* $ and $ c $ such that $ c\delta_*\leq1/2 $ and the following holds for every $ 0 < r^* \leq 1 $ and every $ \omega\in\mathbb{R}^n $ that satisfies \eqref{dio}.
	
	Suppose that $ H(x, y) $ is a real analytic Hamiltonian function defined in the strip
	$ \left| {\operatorname{Im} x} \right| \leq {r^ * },\left| y \right| \leq {r^ * } $, which is of period $ 1 $ in the variables $ {x_1}, \ldots ,{x_n} $ and satisfies
	\begin{align*}
		\left| {H\left( {x,0} \right) - \int_{{\mathbb{T}^n}} {H\left( {\xi ,0} \right)d\xi } } \right| &\leq {\delta ^ * }{r^ * }^{2\tau  + 2},\notag\\
		\left| {{H_y}\left( {x,0} \right) - \omega } \right| &\leq {\delta ^ * }{r^ * }^{\tau  + 1},\notag\\
		\left| {{H_{yy}}\left( {x,y} \right) - Q\left( {x,y} \right)} \right| &\leq \frac{{c{\delta ^ * }}}{{2M}},\notag
	\end{align*}
	for $ \left| {\operatorname{Im} x} \right| \leq r^*,\left| y \right| \leq r^* $, where $ 0 < {\delta ^ * } \leq {\delta _ * } $, and $ Q\left( {x,y} \right) \in {\mathbb{C}^{n \times n}} $ is a symmetric
	(not necessarily analytic) matrix valued function in the strip $ \left| {\operatorname{Im} x} \right| \leq r,\left| y \right| \leq r $
	and satisfies in this domain
	\[\left| {Q\left( z \right)} \right| \leq M,\;\;\left| {{{\left( {\int_{{\mathbb{T}^n}} {Q\left( {x,0} \right)dx} } \right)}^{ - 1}}} \right| \leq M.\]
	Then there exists a real analytic symplectic transformation $ z = \phi \left( \zeta  \right) $ of the
	form
	\[z = \left( {x,y} \right),\;\;\zeta  = \left( {\xi ,\eta } \right),\;\;x = u\left( \xi  \right),\;\;y = v\left( \xi  \right) + u_\xi ^T{\left( \xi  \right)^{ - 1}}\eta \]
	mapping the strip $ \left| {\operatorname{Im} \xi } \right| \leq \theta r^*,\left| \eta  \right| \leq \theta r^* $ into $ \left| {\operatorname{Im} x} \right| \leq r^*,\left| y \right| \leq r^* $, such that $ u\left( \xi  \right) - \xi  $ and $ v\left( \xi  \right) $ are of period $ 1 $ in all variables and the Hamiltonian function $ K: = H \circ \phi  $ satisfies
	\[{K_\xi }\left( {\xi ,0} \right) = 0,\;\;{K_\eta }\left( {\xi ,0} \right) = \omega .\]
	Moreover, $ \phi $ and $ K $ satisfy the estimates
	\begin{align*}
		&\left| {\phi \left( \zeta  \right) - \zeta } \right| \leq c{\delta ^ * }\left( {1 - \theta } \right){r^ * },\;\;\left| {{\phi _\zeta }\left( \zeta  \right) - \mathbb{I}} \right| \leq c{\delta ^ * },\\
		&\left| {{K_{\eta \eta }}\left( \zeta  \right) - Q\left( \zeta  \right)} \right| \leq \frac{{c{\delta ^ * }}}{M},\\
		&\left| {v \circ {u^{ - 1}}\left( x \right)} \right| \leq c{\delta ^ * }{r^ * }^{\tau  + 1},
	\end{align*}
	for $ \left| {\operatorname{Im} \xi } \right| \leq \theta r^*,\left| \eta  \right| \leq \theta r^* $, and $ \left| {\operatorname{Im} x} \right| \leq \theta r^* $.
\end{theorem}

\section*{Acknowledgments}
This work was supported in part by National Basic Research Program of China (Grant No. 2013CB834100), National Natural Science Foundation of China (Grant No. 12071175, Grant No. 11171132,  Grant No. 11571065), Project of Science and Technology Development of Jilin Province (Grant No. 2017C028-1, Grant No. 20190201302JC), and Natural Science Foundation of Jilin Province (Grant No. 20200201253JC).

\end{document}